\documentclass[11pt]{amsart}
\usepackage{amsmath,amssymb}
\usepackage{amsfonts}
\usepackage{amscd}
\usepackage{lscape}
\usepackage{latexsym}
\usepackage{amsfonts,amssymb,stmaryrd,amscd,amsmath,latexsym,amsbsy}

\newcommand{\nc}{\mbox{${\mathbb C}$}}
\newcommand{\nz}{\mbox{${\mathbb Z}$}}
\newcommand{\nq}{\mbox{${\mathbb Q}$}}
\newcommand{\nr}{\mbox{${\mathbb R}$}}
\newcommand{\snc}{\mbox{\scriptsize{${\mathbb C}$}}}
\newcommand{\snr}{\mbox{\scriptsize{${\mathbb R}$}}}
\newcommand{\gtg}{\mathfrak g}
\newcommand{\gtgl}{\mathfrak gl}
\newcommand{\gtsl}{\mathfrak sl}
\newcommand{\gth}{\mathfrak h}
\newcommand{\tf}{\widetilde{f}}
\newcommand{\te}{\widetilde{e}}
\newcommand{\eps}{\varepsilon}
\newcommand{\vphi}{\varphi}
\newcommand{\out}{\mbox{out}}
\newcommand{\sout}{\mbox{\scriptsize{out}}}
\newcommand{\inn}{\mbox{in}}
\newcommand{\sinn}{\mbox{\scriptsize{in}}}
\newcommand{\cB}{\mathcal B}
\newcommand{\cO}{\mathcal O}
\newcommand{\cM}{\mathcal M}

\newcommand{\cV}{\mathcal V}

\newcommand{\onu}{\overline{\nu}}
\newcommand{\ophi}{\overline{\phi}}
\newcommand{\dimc}{\dim_{\mbox{\scriptsize{${\mathbb C}$}}}}
\newcommand{\homc}{\mbox{Hom}_{\mbox{\scriptsize{${\mathbb C}$}}}}

 \newtheorem{thm}{Theorem}[subsection]
 \newtheorem{prop}[thm]{Proposition}
 \newtheorem{lemma}[thm]{Lemma}
 \newtheorem{cor}[thm]{Corollary}
 \newtheorem{defn}[thm]{Definition}

 \newtheorem{ex}[thm]{Example}
 \newtheorem{rem}{Remark}

\newenvironment{Ac}%
 {\hspace*{-1.2em}\textbf{Acknowledgment.}\hspace{0.7em}}{}
 {\hspace*{-1.2em}\textbf{Notation.}\hspace{0.7em}}{}

\title[MV polytopes and crystal basis in type $A$]
{Mirkovi\'c-Vilonen Polytopes and a quiver constructuion of crystal basis in
type $A$}
\author[Y. Saito]{Yoshihisa Saito}
\address{Graduate School of Mathematical Sciences, 
University of Tokyo, 3-8-1 Komaba,
Meguro-ku, Tokyo 153-8914, Japan.}
\email{yosihisa@ms.u-tokyo.ac.jp}
\thanks{
\hspace*{4mm}Reseach of the author is supported by Grant-in-Aid for Scientific 
Resaerch (C) 20540009, Japan Society for the Promontion Science.}

\begin{document}
\bigskip
\begin{abstract}
In the current paper, we give a quiver theoretical interpretation of 
Mirkovi\'c-Vilonen polytopes in type $A_n$. As a by-product, we give a new 
proof of the
Anderson-Mirkovi\'c conjecture which describes the explicit forms of 
the actions of lowering Kashiwara operators on the set of Mirkovi\'c-Vilonen 
polytopes.
%
%
%
\end{abstract}
\maketitle

\section{Introduction}
\subsection{}
Let $B(\infty)$ be the crystal basis of the negative half of the quantum 
universal enveloping algebra associated with a symmetrizable Kac-Moody Lie
algebra $\gtg$. Each element $b\in B(\infty)$ can be written as 
$$b=\tf_{i_1}\tf_{i_2}\cdots\tf_{i_k}b_{\infty}.$$
Here $\tf_i~(i\in I)$ is a lowering Kashiwara operator and $b_{\infty}$ the
highest weight element of $B(\infty)$. However, for a given $b\in B(\infty)$,
the above expression is not unique. For example, the following equality holds
in the case of $\gtg=\gtsl_3$:
$$\tf_1^m\tf_2^{m+n}\tf_1^nb_{\infty}=\tf_2^m\tf_1^{m+n}\tf_2^nb_{\infty}\qquad
(\mbox{for any }m,n\in\nz_{\geq 0}).$$
Therefore, in the study of $B(\infty)$, it is important to give 
\begin{itemize}
\item a parametrization of each element $b\in B(\infty)$ (we call it a 
{\it realization} of $B(\infty)$) and 
\item explicit identifications between several realizations.
\end{itemize}
Until now, many useful realizations of $B(\infty)$ are known. For example, 
\begin{itemize}
\item[(a)] a realization in terms of Lakshimibai-Seshadri path (Littelmann 
\cite{Li}),  
\item[(b)] a polyhedral realization (Nakashima and Zelevinsky \cite{NZ}),
\item[(c)] a Lagrangian (or quiver) realization 
(Kashiwara and the author \cite{KS}), 
{\it etc}.
\end{itemize}
These realizations work for arbitrary symmetrizable Kac-Moody Lie algebras.
On the other hand, for the case that $\gtg$ is a finite dimensional simple
Lie algebras, there is
\begin{itemize}
\item[(d)] a realization by using the theory of PBW basis 
(Lusztig \cite{L1}, the author \cite{S}).  
\end{itemize}
\subsection{}
Recently, Kamnitzer \cite{Kam1}, \cite{Kam2} gave a new realization of 
$B(\infty)$ for the case that $\gtg$ is a finite dimensional simple Lie 
algebras. 

Let us explain a background of his work. In several years ago,  
Mirkovi\'c and Vilonen introduced
a new family of algebraic cycles (called {\it Mirkovi\'c-Vilonen cycles}) in
the corresponding affine Grassmannian \cite{MV1}, \cite{MV2}. Furthermore, 
Breverman and Gaitsgory
\cite{BG} showed that a certain set of Mirkovi\'c-Vilonen cycles has a crystal 
structure which is isomorphic to the crystal basis of an irreducible highest 
weight $U_q(\gtg^{\vee})$-module, where $\gtg^{\vee}$ is the Langlands dual of 
$\gtg$.
By taking the moment map image of Mirkovi\'c-Vilonen cycles, Anderson \cite{A} 
defined a family of convex polytopes in $\gth_{\snr}$, which are called 
{\it Mirkovi\'c-Vilonen}
({\it MV for short}) {\it polytopes}. Here $\gth$ is the Cartan subalgebra of
$\gtg$ and $\gth_{\snr}$ the real form of $\gth$.

After these works, Kamnitzer \cite{Kam1}, 
\cite{Kam2} gave a combinatorial characterization of MV polytopes by using  
a notion of {\it Berenstein-Zelevinsky} ({\it BZ for short}) {\it data}. 
A BZ datum is a certain family of integers (see Section 5 for details), which 
is introduced by Berenstein, Fomin and Zelevinsky \cite{BFZ}. 
Moreover he showed that the set of MV polytopes has a crystal structure which 
is isomorphic to $B(\infty)$. That is, he constructed 
\begin{itemize}
\item[(e)] a realization of $B(\infty)$ in terms of MV polytopes.
\end{itemize}
We remark that he also proved that the above crystal structure on the set
of MV polytopes coincides with the crystal structure coming from one on the
set of Mirkovi\'c-Vilonen cycles, which is introduced in \cite{BG}.
\subsection{}
In the current paper, we focus on the case $\gtg=\gtsl_{n+1}(\nc)$. 
The aim of the paper is to give explicit isomorphisms between the three 
realizations (c), (d) and (e) of $B(\infty)$. In a process for constructing
these isomorphisms, we can also give a quiver theoretical description of 
MV polytopes (or BZ data) in type $A_n$.  
 
As a by-product, we give a new proof of the Anderson-Mirkovi\'c (AM for short)
conjecture. The AM conjecture is a conjecture on the explicit forms of 
the actions of lowering Kashiwara operators on the set of MV polytopes, which
is conjectured by Anderson and Mirkovi\'c (unpublished) and proved by Kamnitzer
\cite{Kam2} (see Theorem \ref{thm:AM}). 
\subsection{}
This paper is organized as follows. In Section 2, we give a quick review on 
the theory of crystals. After recalling basic properties of PBW basis of
the negative half of quantum enveloping algebras in Section 3, we introduce
a crystal structure on PBW basis in Section 4 (see Theorem \ref{thm:PBW}).  
We remark that this is just a reformulation of the result of Reineke \cite{Re}.
In Section 5, after reviewing some of basic facts on MV polytopes following
Kamnitzer \cite{Kam1}, \cite{Kam2}, we construct an isomorphism from
a parametrizing set of PBW basis (so-called Lusztig
data) to the set of MV polytopes in explicit way (see Theorem
\ref{prop:BZ} which will be proved in Section 7). 
In other words, this isomorphism tells us an explicit relation 
between the realization (d) of $B(\infty)$ and (e). In Section 6, we give a 
quiver theoretical interpretation of a BZ datum in type $A$ (see Corollary 
\ref{cor:BZ}). In this consideration, the work of Berenstein, Fomin and 
Zelevinsky 
\cite{BFZ} plays an important role. 
In Section 7, as we mentioned above, we give a proof of 
Theorem \ref{prop:BZ}. In the first half of this section, we give a
short review on a Lagrangian construction of $B(\infty)$, following Kashiwara
and the author \cite{KS}. This is just the realization (c). Since the explicit
isomorphism between the realization (c) and (d) is already known, the problem
can be translated as follows: ``prove that the induced map form the realization
(c) to (e) is an isomorphism of crystals''.  In the second half, we prove
this problem by using the results of Section 6 and quiver theoretical 
considerations. Finally, in Section 8, we give a new proof of the AM 
conjecture in type $A$, as an application of the previous results. 
\subsection{}
Very recently, another quiver theoretical interpretation of BZ data was given 
by Baumann, Kamnitzer and Sadanand (\cite{KamS} for type $A$, and 
\cite{BK} for type $A,D,E$). They gave similar results as our article
(for example, see Theorem 21 in \cite{BK}).
 But their approach is different from ours. Indeed, in their interpretation,
they use the representation theory of preprojective algebras 
(in other words, the double quiver of Dynkin type with certain relations).
On the other hand, in our construction, we only use the ordinary Dynkin 
quiver. In addition to that, as we already mentioned above, we focus only on 
type $A$. By this restriction,
we can get an explicit formula for computing each BZ datum in terms of the 
realization (d) of $B(\infty)$ in type $A$. Consequently, we also have 
a new proof of the AM conjecture in type $A$. Moreover, our approach can be
generalized in affine type $A$ ({\it cf}. \cite{NSS1}, \cite{NSS2}). In other
words, this article is the first step for the above generalization.

\medskip
\begin{Ac}
The author is grateful to Professor Saburo Kakei, Professor Yoshi-\\
yuki Kimura,
Professor Satoshi Naito, Professor Daisuke Sagaki and Professor Yoshihiro 
Takeyama for valuable discussions. 
The author also would like to thank to 
Professor Pierre Baumann
and Professor Joel Kamnitzer for valuable comments on the earlier draft.
\end{Ac}
\section{Preliminaries}
\subsection{Notations}
In this article, we assume $\gtg=\gtsl_{n+1}(\nc)$. Let $\gth$ be the Cartan 
subalgebra of $\gtg$. 
We denote by $\alpha_i\in \gth^*~(i\in I=\{1,2,\cdots,n\})$ the simple roots 
of $\gtg$, and $h_i\in \gth~(i\in I)$ the simple coroot of $\gtg$; note that
$\langle h_i,\alpha_j\rangle=a_{ij}$ for $i,j\in I$, where $\langle \cdot,
\cdot \rangle$ denotes the canonical pairing between $\gth$ and $\gth^*$, 
$(a_{ij})_{i,j\in I}$ the Cartan matrix of type $A_n$. We denote by $P$, $Q$ and 
$\Delta^+$ the weight lattice, the root lattice and the set of all positive 
roots, respectively. Let $W=\mathfrak{S}_{n+1}$
be the Weyl group of $\gtg$. It is generated by simple reflections 
$s_i=s_{\alpha_i}~(i\in I)$. Let $e$ and $w_0$ be the unit element and the 
longest element of $W$, respectively.

Let $U_q=U_q(\gtg)$ be the quantized universal enveloping algebra of type 
$A_n$ with generators $e_i$, $f_i$, $t_i^{\pm}$ $(i\in I)$. 
It is an associative algebra over $\nq(q)$. Let $U_q^-$ be the subalgebra
of $U_q$ generated by $f_i~(i\in I)$.
Define $[l]=\frac{q^l-q^{-l}}{q-q^{-1}}$ and $[k]!=\prod_{l=1}^k[l]$. For 
$x\in U_q(\gtg)$, we denote $x^{(k)}=x^k/[k]!$.
\subsection{Crystals}
\begin{defn}\label{defn:crystal} {\em (1)}
Consider the following data :
\begin{itemize}
\item[(i)] a set $B$,
\item[(ii)] a map $\mbox{\em wt}:B\to P$,
\item[(iii)] maps $\eps_i:B\to \nz\sqcup \{-\infty\}$, 
$\varphi_i:B\to \nz\sqcup \{-\infty\}$ $(i\in I)$,
\item[(iv)] maps $\te_i:B\to B\sqcup \{0\}$, $\tf_i:B\to B\sqcup \{0\}$ 
$(i\in I)$.
\end{itemize}
The sixtuple $(B;\mbox{\em wt},\eps_i,\varphi_i,\te_i,\tf_i)$ {\em ({\em  
denoted by $B$, for short})} is called a crystal if it satisfies
the following axioms: 
\begin{itemize}
\item[(C1)] $\varphi_i(b)=\eps_i(b)+\langle h_i,{\mbox{\em wt}}(b)\rangle.$
\item[(C2)] If $b\in B$ and $\te_ib\in B$, then ${\mbox{\em wt}}(\te_i b)=
{\mbox{\em wt}}(b)+\alpha_i$, $\eps_i(\te_ib)=\eps_i(b)-1$, 
$\varphi_i(\te_ib)=\varphi_i(b)+1$.
\item[(C2')] If $b\in B$ and $\tf_ib\in B$, then ${\mbox{\em wt}}(\tf_i b)=
{\mbox{\em wt}}(b)-\alpha_i$, $\eps_i(\tf_ib)=\eps_i(b)+1$, 
$\varphi_i(\tf_ib)=\varphi_i(b)-1$.
\item[(C3)] For $b,b'\in B$, $b'=\te_i b$ if and only if $b=\tf_i b'$.
\item[(C4)] For $b\in B$, if $\varphi_i(b)=-\infty$, then $\te_ib=\tf_ib=0$.
\end{itemize}
{\em (2)} For two crystals $B_1$ and $B_2$, a morphism $\psi$ from $B_1$ to
$B_2$ is a map $B_1\sqcup \{0\}\to B_2\sqcup\{0\}$ that satisfies the following
conditions:
\begin{itemize}
\item[(i)] $\psi(0)=0$.
\item[(ii)] If $b\in B_1$ and $\psi(b)\in B_2$, then 
${\mbox{\em wt}}(\psi(b))={\mbox{\em wt}}(b)$, $\eps_i(\psi(b))=\eps_i(b)$ and
$\varphi_i(\psi(b))=\varphi_i(b)$.
\item[(iii)] If $b,b'\in B_1$ satisfy $b'=\tf_i(b)$ and $\varphi(b),
\varphi(b')\in B_2$, then $\psi(b')=\tf_i(\psi(b))$.
\end{itemize} 
A morphism $\psi:B_1\to B_2$ is called an isomorphism, if $\psi$ induces an
bijective map $B_1\sqcup \{0\}\to B_2\sqcup\{0\}$ and it commutes with all
$\te_i$ and $\tf_i$. 
\end{defn} 
\subsection{Crystal basis of $U_q^-$}
We shall recall the definition of the crystal basis of $U_q^-$. 
Let $e_i'$ and $e_i''$ be endomorphisms of $U_q^-$ defined by
$$[e_i,x]=\frac{t_ie_i''(x)-t_i^{-1}e_i'(x)}{q-q^{-1}}\quad(x\in U_q^-).$$
It is known that any element $x\in U_q^-$ can be uniquely written as
$$x=\sum_{k\geq 0}f_i^{(k)}x_k\quad\mbox{with }e_i'(x_k)=0.$$
Define modified root operators (so-called Kashiwara operators) 
$\widetilde{e}_i$ and $\widetilde{f}_i$ on $U_q^-$ by
$$\widetilde{e}_ix=\sum_{k\geq 1}f_i^{(k-1)}x_k,\qquad
\widetilde{f}_ix=\sum_{k\geq 0}f_i^{(k+1)}x_k.$$
Let $\mathcal{A}$ be the subring of $\nq(q)$ consisting of rational
functions without a pole at $q=0$. Set
$$L(\infty)=\sum_{k\geq 0,~i_1,\cdots,i_k\in I}\mathcal{A}
\widetilde{f}_{i_1}\cdots \widetilde{f}_{i_k}\cdot 1\subset U_q^-,$$
$$B(\infty)=\left\{\left.\widetilde{f}_{i_1}\cdots \widetilde{f}_{i_k}\cdot 1
\mbox{ mod }qL(\infty)~\right|~k\geq 0,~i_1,\cdots,i_k\in I\right\}.$$
Then the following properties hold:
\vskip 1mm
\noindent
(1) $\widetilde{e}_iL(\infty)\subset L(\infty)$ and 
$\widetilde{f}_iL(\infty)\subset L(\infty)$,
\vskip 1mm
\noindent
(2) $B(\infty)$ is a $\nq$-basis of $L(\infty)/qL(\infty)$,
\vskip 1mm
\noindent
(3) $\widetilde{e}_iB(\infty)\subset B(\infty)\cup \{0\}$ and 
$\widetilde{f}_iB(\infty)\subset B(\infty)$.
\vskip 1mm
\noindent
We call $(L(\infty),B(\infty))$ the crystal basis of $U_q^-$. \\

For $b\in B(\infty)$, we set
$$\mbox{wt}(b)=\mbox{ the weight of $b$},\quad
\eps_i(b)=\mbox{max}\{k\geq 0~|~\te_i^k(b)\ne 0\},\quad
\vphi_i(b)=\eps_i(b)+\langle h_i,\mbox{wt}(b)\rangle.$$
Then $(B(\infty),\mbox{\rm wt},\eps_i,\vphi_i,\te_i,\tf_i)$ is a crystal in the
sense of Definition \ref{defn:crystal}.
\subsection{Orderings on the set of positive roots}
Since $\gtg=\gtsl_{n+1}(\nc)$, any positive root $\beta\in\Delta^+$ can be 
uniquely written as
$$\beta=\sum_{p=i_{\beta}}^{j_{\beta}-1}\alpha_p\quad(\mbox{for some 
$1\leq i_{\beta}<j_{\beta}\leq n+1$}).$$
The correspondence $\beta\mapsto (i_{\beta},j_{\beta})$ defines a bijection 
$\Delta^+\overset{\sim}{\to} \Pi$ where
$$\Pi=\{(i,j)~|~1\leq i<j\leq n+1\}.$$
In the rest of this article, we sometimes identify $\Pi$ with $\Delta^+$ via 
the above bijection.\\

Let $N=n(n+1)/2$ be the length of $w_0$ and fix
$w_0=s_{i_1}s_{i_2}\cdots s_{i_N}~(i_1,i_2,\cdots,i_N\in I)$
a reduced expression of $w_0$. We denote by
${\bf i}=(i_1,i_2,\cdots,i_N)$
the corresponding reduced word. Set
$\beta_k=s_{i_1}s_{i_2}\cdots s_{i_{k-1}}(\alpha_{i_k})~(1\leq k\leq N).$
Then we have $\Pi=\{\beta_1,\beta_2,\cdots,\beta_N\}.$
That is, a reduced word ${\bf i}$ defines
a bijection $\Upsilon_{\bf i}:
\{1,2,\cdots,N\}\overset{\sim}{\to}\Pi$ and it induces
a total ordering $\leq_{~{\bf i}}$ on $\Pi$; 
$$(i_1,j_1)\leq_{~{\bf i}} (i_2,j_2)\quad
\overset{\mbox{def}}{\Longleftrightarrow}\quad
\Upsilon_{\bf i}^{-1}(i_1,j_1)\leq \Upsilon_{\bf i}^{-1} (i_2,j_2).$$ 
%
%
%
%
\begin{ex}{\rm
Let ${\bf i}_0$ be the lexicographically minimal reduced word given by
$${\bf i}_0=(1,2,1,3,2,1,\cdots,n,n-1,\cdots,1).$$
Then we have
$$(i_1,j_1)\leq_{{\bf i}_0} (i_2,j_2)\quad\mbox{if and only if}\quad
\left\{\begin{array}{ll}
j_1<j_2\\
\mbox{or}\\
j_1=j_2~\mbox{and}~i_1\leq i_2.
\end{array}\right.$$
}\end{ex}
\section{PBW basis and Lusztig data associated to reduced words}
\subsection{PBW basis of quantized universal enveloping algebras}
For $i\in I$, introduce an 
$\nq(q)$-algebra automorphism $T_i$ of $U_q$ as :
$$T_i(e_j):=\left\{\begin{array}{ll}
-f_it_i & (i=j),\\
\sum_{k=0}^{-a_{ij}}(-1)^{a_{ij}+k}q^{a_{ij}+k}e_i^{(k)}e_je_i^{(-a_{ij}-k)} & 
(i\ne j),
\end{array}\right.$$
$$T_i(f_j)=\left\{\begin{array}{ll}
-t_i^{-1}e_i, & (i=j),\\
\sum_{k=0}^{-a_{ij}}(-1)^{a_{ij}+k}q^{-a_{ij}-k}f_i^{(-a_{ij}-k)}f_jf_i^{(k)},& 
(i\ne j),
\end{array}
\right.$$
$$T_i(t_j)=t_jt_i^{-a_{ij}}.$$
It is known that these operators satisfy the braid relations:
$$T_iT_jT_i=T_jT_iT_j~(a_{ij}=-1),\qquad T_iT_j=T_jT_i~(a_{ij}=0).$$

For a reduced word ${\bf i}=(i_1,i_2,\cdots,i_N)$ 
and ${\bf c}=(c_1,\cdots,c_N)\in\nz_{\geq 0}^N$, we define
$$P_{\bf i}({\bf c})=f_{i_1}^{(c_1)}
\left(T_{i_1}(f_{i_1}^{(c_2)})\right)\cdots 
\left(T_{i_1}T_{i_2}\cdots 
T_{i_{N-1}}(f_{i_N}^{(c_N)})\right)$$
and 
$$B_{\bf i}=\left\{P_{\bf i}({\bf c})~|~{\bf c}\in \nz_{\geq 0}^N\right\},
\qquad L_{\bf i}=\sum_{{\bf c}}\mathcal{A}P_{\bf i}({\bf c}).$$

\begin{prop}[\cite{L1},\cite{S}]
{\rm (1)} $B_{\bf i}$ is a $\nq(q)$-basis of $U_q^-$.
\vskip 1mm
\noindent
{\rm (2)} $L_{\bf i}=L(\infty)$.
Moreover $B_{\bf i}$ is a free $\mathcal{A}$-basis of $L(\infty)$.
\vskip 1mm
\noindent
{\rm (3)} $B_{\bf i}\equiv B(\infty)~\mbox{mod}~qL(\infty)$. 
\end{prop}
\begin{defn}
For a giving reduced word ${\bf i}$, the basis 
$B_{\bf i}=\left\{P_{\bf i}({\bf c})~|~{\bf c}\in \nz_{\geq 0}^N\right\}$ is 
called {\it the PBW basis of $U_q^-$ associated to a reduced word ${\bf i}$}.
There is a bijection $\Xi_{{\bf i}}:\nz_{\geq 0}^N\overset{\sim}{\to} B(\infty)$
defined by
${\bf c}\mapsto P_{\bf i}({\bf c}) ~\mbox{mod}~qL(\infty)$. For 
$b\in B(\infty)$, we call $\Xi_{{\bf i}}^{-1}(b)\in \nz_{\geq 0}^N$ the
${\bf i}$-Lusztig datum of $b\in B(\infty)$.
\end{defn}
\subsection{The transition maps}
For a reduced word ${\bf i}$, consider
$$\mathcal{B}^{\bf i}=
\left\{{\bf a^{\bf i}}=(a_{i,j}^{\bf i})_{(i,j)\in\Pi}~|~
a_{i,j}^{\bf i}\in \nz_{\geq 0}\mbox{ for any }(i,j)\in\Pi\right\}$$
the set of all $N$-tuples of non-negative integers indexed by $\Pi$. 
From now on, we regard $\cB^{\bf i}$ as the set of all 
${\bf i}$-Lusztig data via the bijection 
$\nz_{\geq 0}^N\overset{\sim}{\to} \cB^{\bf i}$ induced from 
$\Upsilon_{\bf i}:\{1,\cdots,N\}\overset{\sim}{\to}\Pi$. 
For two reduced words ${\bf i}$ and ${\bf i}'$, let us consider the following
composition of the bijections:
$$R_{\bf i}^{{\bf i}'}=\Xi_{{\bf i}'}^{-1}\circ\Xi_{\bf i}:
\cB^{\bf i}\overset{\sim}{\to}B(\infty)
\overset{\sim}{\to}\cB^{{\bf i}'}.$$
We call $R_{\bf i}^{{\bf i}'}$ the {\it transition map form ${\bf i}$ to 
${\bf i}'$}. The explicit form of $R_{\bf i}^{{\bf i}'}$ is known (\cite{BFZ}),
but we omit to give it in this article. 

For an arbitrarily reduced word ${\bf i}$, the set of all ${\bf i}$-Lusztig 
data $\cB^{\bf i}$ has a crystal structure
which is induced form the bijection $\Xi_{{\bf i}}:\cB^{\bf i}
\overset{\sim}{\to} B(\infty)$. 
Especially, for the lexicographically minimal reduced word ${\bf i}_0$, 
we denote $\cB=\cB^{{\bf i}_0}$ and ${\bf a}={\bf a}^{{\bf i}_0}\in \cB$.
It has a central role in this article.
In the next section we will give the crystal structure on $\cB=\cB^{{\bf i}_0}$ 
in explicit way.
\subsection{$\ast$-structure}
Define a $\nq(q)$-algebra anti-involution $\ast$ of $U_q$ by
$$e_i^*=e_i,\quad f_i^*=f_i,\quad (t_i^{\pm})^*=t_i^{\mp}\qquad (i\in I).$$
By the construction, it is easy to see $L(\infty)^*=L(\infty)$. Therefore
$\ast$ induces a $\nq$-linear automorphism of $L(\infty)/qL(\infty)$. Moreover
the following theorem is known.
\begin{thm}[\cite{K3}]\label{thm:ast}
The anti-involution $\ast$ induces an involutive bijection on $B(\infty)$.
\end{thm}
By the above theorem, we can define the operators $\te_i^*$, $\tf_i^*$ on 
$B(\infty)$ by
$$\te_i^*=\ast \te_i \ast,\quad \tf_i^*=\ast \tf_i \ast.$$
By the definition, it is obvious that $b_{\infty}^*=b_{\infty}$. 
Here $b_{\infty}$ is the highest weight element of $B(\infty)$.
Therefore we have
$$\eps_i(b^*)=\mbox{max}\{k\geq 0~|~(\te_i^*)^k(b)\ne 0\}$$
for $b\in B(\infty)$. From now on, we denote $\eps_i^*(b)=\eps_i(b^*)$.\\

Define a $\nq(q)$-algebra automorphism $T_i^*~(i\in I)$ by
$$T_i^*=\ast\circ T_i\circ\ast.$$
Then we have
\begin{align*}
P_{\bf i}^*({\bf c})&=
\left(f_{i_1}^{(c_1)}\left(T_{i_1}(f_{i_1}^{(c_2)})\right)\cdots 
\left(T_{i_1}T_{i_2}\cdots 
T_{i_{N-1}}(f_{i_N}^{(c_N)})\right)\right)^*\\
&=\left(T_{i_1}^*T_{i_2}^*\cdots T_{i_{N-1}}^*(f_{i_N}^{(c_N)})\right)
\cdots \left(T_{i_1}^*(f_{i_1}^{(c_2)})\right)f_{i_1}.
\end{align*}
Set
$$B^*_{\bf i}=\left\{\left.P_{\bf i}^*({\bf c})~\right|
{\bf c}\in \nz_{\geq 0}^N\right\}.$$
It also gives a $\nq(q)$-basis of $U_q^-$. We call it the {\it $\ast$-PBW 
basis of $U_q^-$ associated to a reduced word ${\bf i}$}.\\
 
By Theorem \ref{thm:ast}, we have the following corollary.
\begin{cor}
$$\left\{\left.P_{\bf i}^*({\bf c})~\mbox{mod}~qL(\infty)~\right|
{\bf c}\in \nz_{\geq 0}^N\right\}=B(\infty).$$
\end{cor} 
\section{Crystal structure on ${\bf i}_0$-Lusztig data}
\subsection{Definition of crystal structures}

We shall define two crystal structures on the set of all ${\bf i}_0$-Lusztig
datum $\cB$. For ${\bf a}\in \cB$, define the weight $\mbox{wt}({\bf a})$
of ${\bf a}$ by 
$$\mbox{wt}({\bf a})=-\sum_{i\in I}m_i\alpha_i,\quad\mbox{where}\quad
m_i=\sum_{k=1}^i\sum_{l=i+1}^{n+1}a_{k,l}\quad (i\in I).$$
For $i\in I$, set
$$A^{(i)}_k({\bf a})=\sum_{s=1}^k(a_{s,i+1}-a_{s-1,i})\quad (1\leq k\leq i),$$
$$A^{\ast (i)}_l({\bf a})=\sum_{t=l+1}^{n+1}(a_{i,t}-a_{i+1,t+1})\quad 
(i\leq l\leq n).$$
Here we set $a_{0,i}=0$ and $a_{i+1,n+2}=0$. Define 
$$\eps_i({\bf a})=
\mbox{max}\left\{A_1^{(i)}({\bf a}),\cdots,A_i^{(i)}({\bf a})\right\},
\quad\vphi_i({\bf a}) = \eps_i({\bf a})+\langle h_i,\mbox{wt}({\bf a})
\rangle,$$
$$\eps_i^*({\bf a})=\mbox{max}\left\{A^{\ast (i)}_i({\bf a}),\cdots,
A^{\ast (i)}_n({\bf a})\right\},\quad
\varphi_i^*({\bf a})=\eps_i^*({\bf a})+\langle h_i,\mbox{wt}({\bf a})
\rangle.$$
Let
$$k_e=\mbox{min}\left\{1\leq k\leq i\left|~\eps_i({\bf a})
=A_k^{(i)}({\bf a})\right.\right\},\quad
k_f=\mbox{max}\left\{1\leq k\leq i\left|~\eps_i({\bf a})
=A_k^{(i)}({\bf a})\right.\right\},$$
$$l_e=\mbox{max}\left\{i\leq l\leq n\left|~\eps_i^*({\bf a})
=A_l^{\ast (i)}({\bf a})\right.\right\},~
l_f=\mbox{min}\left\{i\leq l\leq n\left|~\eps_i^*({\bf a})
=A_l^{\ast (i)}({\bf a})\right.\right\}.$$
For a given ${\bf a}\in \cB$, we introduce four $N$-tuples of integers 
${\bf a}^{(p)}=\left(a_{k,l}^{(p)}\right)~(p=1,2,3,4)$ by 
\begin{align*}
{a}^{(1)}_{k,l}&=\left\{\begin{array}{ll}
a_{k_e,i}+1 & (k=k_e,~l=i),\\
a_{k_e,i+1}-1 & (k=k_e,~l=i+1),\\
a_{k,l} & (\mbox{otherwise}).
\end{array}\right.\\
{a}_{k,l}^{(2)}&=\left\{\begin{array}{ll}
a_{k_f,i}-1 & (k=k_f,~l=i),\\
a_{k_f,i+1}+1 & (k=k_f,~l=i+1),\\
a_{k,l} & (\mbox{otherwise}),
\end{array}\right.\\
{a}_{k,l}^{(3)}&=\left\{\begin{array}{ll}
a_{i,l_e+1}-1 & (k=i,~l=l_e+1),\\
a_{i+1,l_e+1}+1 & (k=i+1,~l=l_e+1),\\
a_{k,l} & (\mbox{otherwise}).
\end{array}\right.\\
{a}^{(4)}_{k,l}&=\left\{\begin{array}{ll}
a_{i,l_f+1}+1 & (k=i,~l=l_f+1),\\
a_{i+1,l_f+1}-1 & (k=i+1,~l=l_f+1),\\
a_{k,l} & (\mbox{otherwise}).
\end{array}\right.
\end{align*}

\begin{lemma}\label{lemma:well-def}
{\rm (1)} For any ${\bf a}\in \cB$ with $\eps_i({\bf a})>0$, ${\bf a}^{(1)}$
is an $N$-tuple of non-negative integers. In other words, ${\bf a}^{(1)}$ is
an element of $\cB$.\\ 
{\rm (2)} For any ${\bf a}\in \cB$ with $\eps_i^*({\bf a})>0$, ${\bf a}^{(3)}$
is an element of $\cB$.\\
{\rm (3)} For any ${\bf a}\in \cB$, both ${\bf a}^{(2)}$ and ${\bf a}^{(4)}$
are elements of $\cB$.
\end{lemma}
\begin{proof}
We only give a proof of (1). It is enough to show that $a_{k_e,i+1}>0$. 
If $k_e=1$, we have 
$a_{k_e,i+1}=A_{1}^{(i)}({\bf a})=\eps_i({\bf a})>0$. 
Assume $k_e>1$. Then, by the definition of $k_e$, we have 
$A_{k_e-1}^{(i)}({\bf a}) < A_{k_e}^{(i)}({\bf a})$.
Therefore $A_{k_e}^{(i)}({\bf a})-A_{k_e-1}^{(i)}({\bf a})
=a_{{k_e},i+1}-a_{{k_e-1},i}>0$. Since $a_{{k_e-1},i}\geq 0$, we have
$a_{{k_e},i+1}>0$.
\end{proof}
Now we define Kashiwara operators on $\cB$ as:
$$\te_i{\bf a}=\left\{\begin{array}{ll}
0&(\mbox{if }\eps_i({\bf a})=0),\\
{\bf a}^{(1)}&(\mbox{if }\eps_i({\bf a})>0),
\end{array}\right.\quad \tf_i{\bf a}={\bf a}^{(2)},$$
$$\te_i^*{\bf a}=\left\{\begin{array}{ll}
0&(\mbox{if }\eps_i^*({\bf a})=0),\\
{\bf a}^{(3)}&(\mbox{if }\eps_i^*({\bf a})>0),
\end{array}\right.\quad \tf_i^*{\bf a}={\bf a}^{(4)}.$$

\begin{prop}
{\rm (1)} $(\cB,\mbox{\rm wt},\eps_i,\vphi_i,\te_i,\tf_i)$ is a crystal 
in the sense of Definition \ref{defn:crystal}.\\
{\rm (2)} $(\cB,\mbox{\rm wt},\eps_i^*,\vphi_i^*,\te_i^*,\tf_i^*)$ is 
a crystal 
in the sense of Definition \ref{defn:crystal}.
\end{prop} 
From the definition, one can easily check the axiom (C1) $\sim$ (C4). So we omit
to give a detail.
\subsection{Crystal structure on PBW basis associated to ${\bf i}_0$}
As we mentioned in the previous subsection, we regard $\cB$ as the set of 
${\bf i}_0$-Lusztig datum and denote by 
$\{P_{{\bf i}_0}({\bf a})~|~{\bf a}\in\cB\}$ the corresponding PBW basis.
\begin{thm}\label{thm:PBW}
{\rm (1)} We have
$$\te_iP_{{\bf i}_0}({\bf a})\equiv P_{{\bf i}_0}(\te_i{\bf a})~
\mbox{\rm mod }qL(\infty)\quad\mbox{\rm and}\quad
\tf_iP_{{\bf i}_0}({\bf a})\equiv P_{{\bf i}_0}(\tf_i{\bf a})~
\mbox{\rm mod }qL(\infty).$$ 
{\rm (2)} We have
$$\te_iP_{{\bf i}_0}^*({\bf a})\equiv P_{{\bf i}_0}^*(\te_i^*{\bf a})~
\mbox{\rm mod }qL(\infty)\quad\mbox{\rm and}\quad
\tf_iP_{{\bf i}_0}^*({\bf a})\equiv P_{{\bf i}_0}^*(\tf_i^*{\bf a})~
\mbox{\rm mod }qL(\infty).$$ 
\end{thm}
\begin{rem}{\rm
The formulas (2) is proved by Reineke \cite{Re} (see also \cite{Sav}, 
\cite{EK}). Note that in \cite{Re},
he denotes our $\te_i^*$ and $\tf_i^*$ on $\cB$, by $\te_i$ and $\tf_i$, 
respectively. The formulas (1) is proved by the similar method. 
}\end{rem}
By the definition, we immediately have the next corollary.
\begin{cor} {\rm (1)} We have
$$\te_i^*P_{{\bf i}_0}({\bf a})\equiv P_{{\bf i}_0}(\te_i^*{\bf a})~
\mbox{\rm mod }qL(\infty)\quad\mbox{\rm and}\quad
\tf_i^*P_{{\bf i}_0}({\bf a})\equiv P_{{\bf i}_0}(\tf_i^*{\bf a})~
\mbox{\rm mod }qL(\infty).$$
{\rm (2)} As a by-product, we have that $\cB$ is isomorphic to $B(\infty)$
as a crystal with $\ast$-structure.
\end{cor}
\section{Mirkovi\'c-Vilonen polytopes in type $A_n$}
\subsection{Mirkovi\'c-Vilonen polytopes and Berenstein-Zelevinsky data}
Let $\Lambda_i~(i\in I)$ be a fundamental weight for $\gtg$. Set
$$\Gamma_n=\bigcup_{i\in I}W\Lambda_i$$
and an element $\gamma\in \Gamma_n$ is called a chamber weight. 
Let ${\bf M}=(M_{\gamma})_{\gamma\in \Gamma_n}$ be a collection of
integers indexed by $\Gamma_n$. For each $\gamma\in\Gamma_n$, we call
$M_{\gamma}$ the the $\gamma$-component of ${\bf M}$, and denote it by 
$({\bf M})_{\gamma}$.

For a given ${\bf M}=(M_{\gamma})_{\gamma\in \Gamma_n}$, consider the following
polytope in $\gth_{\snr}$:
$$P({\bf M})=\{h\in \gth_{\snr}~|~\langle h,\gamma\rangle\geq M_{\gamma}~
(\forall \gamma\in\Gamma_n)\}.$$
\begin{defn}
{\rm (1)} A polytope $P({\bf M})$
is called a pseudo-Weyl polytope if it satisfies the following condition:
\vskip 1mm
\noindent
{\rm (BZ-1)} {\em ({\em edge inequalities})} for all $w\in W$ and $i\in I$,
$$M_{w\Lambda_i}+M_{ws_i\Lambda_i}
+\sum_{j\in I\setminus\{i\}}a_{ji}M_{w\Lambda_j}
\leq 0;$$
\noindent
{\rm (2)} A pseudo-Weyl polytope $P({\bf M})$ is called a Mirkovi\'c-Vilonen
{\em ({\em MV for short})} polytope if it satisfies the following condition:
\vskip 1mm
\noindent
{\rm (BZ-2)} {\em ({\em $3$-term relations})}
for every $w\in W$ and $i,j\in I$ with $a_{ij}=a_{ji}=-1$ and $ws_i>w$, $ws_j>w$,
$$M_{ws_i\Lambda_i}+M_{ws_j\Lambda_j}=\mbox{\rm min}\left\{
M_{w\Lambda_i}+M_{ws_is_j\Lambda_j},~M_{w\Lambda_j}
+M_{ws_js_i\Lambda_i}
\right\}.$$
Here $(a_{ij})_{i,j\in I}$ is the Cartan matrix of type $A_n$ and $>$ is the 
strong Bruhat ordering of $W$. If $P({\bf M})$ is a MV polytope, the 
corresponding collection of integers ${\bf M}=(M_{\gamma})_{\gamma\in \Gamma_n}$
is called Berenstein-Zelevinsky {\em ({\em BZ for short})} datum of type $A_n$.
\end{defn} 
\begin{rem}{\rm
For a collection of integers ${\bf M}=(M_{\gamma})_{\gamma\in \Gamma_n}$ 
which satisfies the condition (BZ-1), set
$$\mu_{w}=\sum_{i\in I}M_{w\Lambda_i}w h_i\in 
\gth_{\mbox{\scriptsize $\nr$}}\qquad(w\in W)$$
and consider a collection of vectors  
$$\mu_{\bullet}=(\mu_w)_{w\in W}\subset \gth_{\mbox{\scriptsize $\nr$}}.$$
Then Kamnitzer \cite{Kam2} showed the corresponding pseudo-Weyl polytope 
$P({\bf M})$ is the convex hull of $\mu_{\bullet}$. 
That is, there is a one to one correspondence between the set of pseudo-Weyl 
polytopes and the set of collections of integers which satisfies the condition 
(BZ-1).
}\end{rem}
\begin{defn}
A BZ datum ${\bf M}^{w_0}=(M_{\gamma}^{w_0})_{\gamma\in \Gamma_n}$
is called a $w_0$-BZ datum of type $A_n$ if it satisfies 
\vskip 1mm
\noindent
{\rm (BZ-0)} {\em ({\em $w_0$-normalization condition})} for all $i\in I$, 
$$M_{w_0\Lambda_i}^{w_0}=0.$$
We denote by $\mathcal{BZ}^{w_0}$ the set of all $w_0$-BZ data.
\end{defn}
A set of integers $K\subset [1,n+1]$ will be called a Maya diagram of rank $n$.
We denote by $\cM_n$ the set of all Maya diagram of rank $n$. Set
$\cM^{\times}_n=\cM_n\setminus \left\{\phi,[1,n+1]\right\}$. From now on, we 
identify the set of chamber weights $\Gamma_n$ with $\cM^{\times}_n$ by the
following way: recall that there is a natural action of 
$W\cong \mathfrak{S}_{n+1}$ on the set $\{1,2,\cdots,n+1\}$. Consider the map
$\Gamma_n\to \cM_{n}^{\times}$ defined by $\gamma=w\Lambda_i\mapsto w\cdot[1,i]$.
Since this map is bijective, we can identify $\Gamma_n$ with $\cM^{\times}_n$.
By the above identification, $\Lambda_i$ and $w_0\Lambda_i$ 
are regarded as
$$\Lambda_i~\leftrightarrow~[1,i],\quad w_0\Lambda_i~\leftrightarrow~
[n-i+2,n+1]\qquad(i\in I).$$
Under the above identification, the definition of $w_0$-BZ datum can be 
rewritten as follows:
\begin{lemma}
A collection ${\bf M}^{w_0}=(M_{K}^{w_0})_{K\in \cM_n^{\times}}$ of integers
is a $w_0$-BZ datum of type $A_n$ if and only if it satisfies the 
following conditions:
\vskip 1mm
\noindent
{\rm (BZ-0)'} for all $i\in I$, 
$$M_{[n-i+2,n+1]}^{w_0}=0;$$
\noindent
{\rm (BZ-1)'} for every two indices $i\ne j$ in $[1,n+1]$ and every 
$K\in \cM_n$ 
with $K\cap\{i,j\}=\phi$, 
$$M_{Ki}^{w_0}+M_{Kj}^{w_0}\leq M_{Kij}^{w_0}+M_{K}^{w_0};$$
\noindent
{\rm (BZ-2)'} for every three indices $i<j<k$ in $[1,n+1]$ and every 
$K\in \cM_n$ with $K\cap\{i,j,k\}=\phi$, 
$$M_{Kik}^{w_0}+M_{Mj}^{w_0}=\mbox{\rm min}\left\{
M_{Kij}^{w_0}+M_{Kk}^{w_0},~M_{Kjk}^{w_0}+M_{Ki}^{w_0}\right\}.$$
Here we denote $M_{Ki}^{w_0}=M_{K\cup\{i\}}^{w_0}$, etc., and set 
$M_{\phi}^{w_0}=M_{[1,n+1]}^{w_0}=0$.  
\end{lemma}
\begin{rem}{\rm
The conditions (BZ-2)' are just the conditions which are called the 
$3$-term relations in \cite{BFZ}.
}\end{rem}
\subsection{$e$-BZ datum}
\begin{defn}
A collection ${\bf M}^{e}=(M_{K}^{e})_{K\in \cM_n^{\times}}$ of integers is called 
a $e$-normalized Berenstein-Zelevinsky {\em ({\em $e$-BZ for short})} 
datum of type $A_n$ if it satisfies the above {\em (BZ-1)'}, {\em (BZ-2)'} and 
\vskip 1mm
\noindent
{\rm (BZ-0)''} {\em ({\em $e$-normalization condition})} for all $i\in I$, 
$$M_{[1,i]}^{e}=0.$$
We denote by $\mathcal{BZ}^{e}$ the set of all $e$-BZ data.
\end{defn}

For $K\in\cM_n^{\times}$, let $K^c=[1,n+1]\setminus K$ be the complement
of $K$ in $[1,n+1]$. For ${\bf M}^{w_0}=(M_K^{w_0})_{K\in \cM_n^{\times}}\in 
\mathcal{BZ}^{w_0}$, we 
define a new collection of integers 
${\bf M}^{w_0\ast}=(M_K^{w_0\ast})_{K\in \cM_n^{\times}}$ by
$$M_K^{w_0\ast}=M_{K^c}^{w_0}\quad (K\in \cM_n^{\times}).$$
\begin{lemma}
The map $\ast:{\bf M}^{w_0}\mapsto {\bf M}^{w_0\ast}$ defines a bijection form 
$\mathcal{BZ}^{w_0}$ to $\mathcal{BZ}^{e}$. The inverse $\mathcal{BZ}^{e}\to
\mathcal{BZ}^{w_0}$ of the map $\ast$ is given by 
$${\bf M}^{e}=(M_K^e)\mapsto {\bf M}^{e\ast}=(M_K^{e\ast}),\quad \mbox{where }
M_K^{e\ast}=M_{K^c}^e~(K\in\cM_n^{\times}).$$
\end{lemma}
\begin{proof}
Let ${\bf M}^{w_0}=(M_K^{w_0})\in \mathcal{BZ}_{w_0}$. Then it is clear
that the collection of integers ${\bf M}^{w_0\ast}=(M_K^{w_0\ast})$ satisfies
(BZ-0)''. Let us prove that ${\bf M}^{w_0\ast}$ satisfies 
(BZ-1)'. Let $i\ne j$ be two indices in $[1,n+1]$ and $K\in \cM_n$ 
with $K\cap\{i,j\}=\phi$. For such $i,j$ and $K$, we set 
$L=K^c\setminus\{i,j\}$. Then we have $L\in \cM_n$ and 
$L\cap \{i,j\}=\phi$. Since ${\bf M}^{w_0}$ satisfies (BZ-1)', we have
$$M_{Li}^{w_0}+M_{Lj}^{w_0}\leq M_{Lij}^{w_0}+M_{L}^{w_0}.$$ 
Because
$$K^c=Lij,\quad (Ki)^c=Lj,\quad (Kj)^c=Li,\quad (Kij)^c=L,$$
we have
$$M_{K}^{w_0\ast}=M^{w_0}_{Lij},\quad M_{Ki}^{w_0\ast}=M^{w_0}_{Lj},
\quad M_{Kj}^{w_0\ast}=M^{w_0}_{Li},\quad M_{Kij}^{w_0\ast}=M^{w_0}_{L}.$$
Therefore we have
$$M^{w_0\ast}_{Kj}+M_{Ki}^{w_0\ast}\leq M^{w_0\ast}_{K}+M^{w_0\ast}_{Kij}.$$
This is nothing but (BZ-1)' for ${\bf M}^{w_0\ast}$. 

By the similar argument we can check (BZ-2)' for ${\bf M}^{w_0\ast}$. Thus, 
${\bf M}^{w_0\ast}$ is an $e$-BZ datum. The other statements are clear by the 
construction. 
\end{proof}
\subsection{Crystal structure on $w_0$-BZ data}
We denote $\mathcal{MV}=\left\{P({\bf M}^{w_0})~|~{\bf M}^{w_0}
\in \mathcal{BZ}^{w_0}\right\}$. In \cite{Kam2}, Kamnitzer defines 
a crystal structure on 
$\mathcal{MV}$ and shows it is isomorphic to $B(\infty)$ as a crystal. 
Since the map $\mathcal{BZ}^{w_0}\to \mathcal{MV}$ defined by 
${\bf M}^{w_0}\mapsto P({\bf M}^{w_0})$ is bijective, we can define a crystal 
structure on 
$\mathcal{BZ}^{w_0}$ in such a way that the above bijection gives an 
isomorphism of crystals. In the following, we recall the description of 
this crystal structure on $\mathcal{BZ}^{w_0}$ form \cite{Kam2}. 
\begin{rem}{\rm
In \cite{Kam2}, he uses the set of chamber weights $\Gamma_n$ as 
the index set of $\mathcal{BZ}^{w_0}$.
But, for later use, we will reformulate the above crystal structure on 
$\mathcal{BZ}^{w_0}$ by using the set of Maya diagrams $\cM_n^{\times}$ 
instead of $\Gamma_n$.  
}\end{rem}

Let ${\bf M}^{w_0}=(M_K^{w_0})\in \mathcal{BZ}^{w_0}$. Define
the weight $\mbox{wt}({\bf M}^{w_0})$ of ${\bf M}^{w_0}$ by 
$$\mbox{wt}({\bf M}^{w_0})=\sum_{i\in I}M_{[1,i]}^{w_0}\alpha_i.$$
For $i\in I$, we set 
\begin{align*}
\eps_i({\bf M}^{w_0})&=-\left(M_{[1,i]}^{w_0}+M_{[1,i+1]\setminus\{i\}}^{w_0}
-M_{[1,i+1]}^{w_0}-M_{[1,i]\setminus\{i\}}^{w_0}\right),\\
\vphi_i({\bf M}^{w_0})&=\eps_i({\bf M}^{w_0})
+\langle h_i,\mbox{wt}({\bf M}^{w_0})\rangle.
\end{align*}
We remark that $\eps_i({\bf M}^{w_0})$ is a non-negative integer in view of 
(BZ-1)'.

Let us define the action of Kashiwara operators $\te_i$ and $\tf_i~(i\in I)$.
We recall the following fact due to Kamnitzer:
\begin{prop}[\cite{Kam2}]\label{prop:Kas-op}
Let ${\bf M}^{w_0}=(M_K^{w_0})\in \mathcal{BZ}^{w_0}$ be a $w_0$-BZ 
datum.
\vskip 1mm
\noindent
{\rm (1)} If $\eps_i({\bf M}^{w_0})>0$, there exists a unique $w_0$-BZ 
datum which is denoted by 
$\te_i{\bf M}^{w_0}$ such that
\begin{itemize}
\item[(i)] $(\te_i{\bf M}^{w_0})_{[1,i]}=M_{[1,i]}^{w_0}+1$,
\item[(ii)] $(\te_i{\bf M}^{w_0})_K=M_K^{w_0}$ for all 
$K\in \cM_n^{\times}\setminus \cM_n^{\times}(i)$.
\end{itemize}
Here $\cM_n^{\times}(i)=\left\{K\in \cM_n^{\times}~|~i\in K\mbox{ and }
i+1\not\in K\right\}\subset \cM_n^{\times}$.
\vskip 1mm
\noindent
{\rm (2)} There exists a unique a unique $w_0$-BZ datum which is denoted by 
$\tf_i{\bf M}^{w_0}$ such that
\begin{itemize}
\item[(iii)] $(\tf_i{\bf M}^{w_0})_{[1,i]}=M_{[1,i]}^{w_0}-1$,
\item[(iv)] $(\tf_i{\bf M}^{w_0})_K=M_K^{w_0}$ for all 
$K\in \cM_n^{\times}\setminus \cM_n^{\times}(i)$.
\end{itemize}
\end{prop}
If $\eps_i({\bf M}^{w_0})=0$, we set $\te_i{\bf M}^{w_0}=0$.
\begin{thm}[\cite{Kam2}]\label{thm:BZ-crys}
$(\mathcal{BZ}^{w_0},\mbox{\em wt},\eps_i,\vphi_i,\te_i,\tf_i)$ is a crystal
in the sense of Definition \ref{defn:crystal}, which is isomorphic to 
$B(\infty)$. 
\end{thm}
\subsection{$\ast$-crystal structure on $e$-BZ data}
By using Theorem \ref{thm:BZ-crys}, we can define a crystal structure on the
set of $e$-BZ data $\mathcal{BZ}^{e}$ as follows. Recall the bijection 
$\ast:\mathcal{BZ}^{w_0}\overset{\sim}{\to} \mathcal{BZ}^{e}$ and its inverse
which is also denoted by $\ast$. For ${\bf M}^e\in \mathcal{BZ}^{e}$, we set
$$\mbox{wt}({\bf M}^e)=\mbox{wt}({\bf M}^{e\ast}),\quad 
\eps_i^*({\bf M}^e)=\eps_i({\bf M}^{e\ast}),\quad
\vphi_i^*({\bf M}^e)=\vphi_i({\bf M}^{e\ast}).$$
Here we remark that ${\bf M}^{e\ast}$ is a $w_0$-BZ datum 
and the right hand sides are already defined. The Kashiwara operators 
$\te_i^*,~\tf_i^*~(i\in I)$ on $\mathcal{BZ}^{e}$ are defined by 
$$\te_i^*=\ast\circ\te_i\circ\ast,\quad \tf_i^*=\ast\circ\tf_i\circ\ast.$$
The following corollary is an easy consequence of Proposition \ref{prop:Kas-op}
and Theorem \ref{thm:BZ-crys}.
\begin{cor}\label{cor:ast-crys}
{\rm (1)} Let ${\bf M}^{e}=(M_K^{e})\in \mathcal{BZ}^{e}$ be an $e$-BZ 
datum. If $\eps_i^*({\bf M}^{e})>0$, there exists a unique $e$-BZ 
datum which is denoted by 
$\te_i^*{\bf M}^{e}$ such that
\begin{itemize}
\item[(i)] $(\te_i^*{\bf M}^{e})_{[1,i]^c}=M_{[1,i]^c}^{e}+1$,
\item[(ii)] $(\te_i^*{\bf M}^{e})_K=M_K^{e}$ for all 
$K\in \cM_n^{\times}\setminus \cM_n^{\times}(i)^{\ast}$. 
\end{itemize}
Here $\cM_n^{\times}(i)^*=\left\{K\in \cM_n^{\times}~|~i\not\in K\mbox{ and }
i+1\in K\right\}\subset \cM_n^{\times}$.
\vskip 1mm
\noindent
{\rm (2)} There exists a unique a unique $e$-BZ datum which is denoted by 
$\tf_i^*{\bf M}^{e}$ such that
\begin{itemize}
\item[(iii)] $(\tf_i^*{\bf M}^{e})_{[1,i]^c}=M_{[1,i]^c}^{e}-1$,
\item[(iv)] $(\tf_i^*{\bf M}^{e})_K=M_K^{e}$ for all 
$K\in \cM_n^{\times}\setminus \cM_n^{\times}(i)^*$.
\end{itemize}
\vskip 1mm
\noindent
{\rm (3)} $(\mathcal{BZ}^{e},\mbox{\em wt},\eps_i^*,\vphi_i^*,\te_i^*,\tf_i^*)$ 
is a crystal in the sense of Definition \ref{defn:crystal}, 
which is isomorphic to $B(\infty)$.
\end{cor}
\subsection{Anderson-Mirkovi\'c conjecture}\label{conj:AM}
Let ${\bf M}^{w_0}=(M_K^{w_0})\in \mathcal{BZ}^{w_0}$ be be a $w_0$-BZ datum.
In \cite{Kam2}, Kamnitzer gives the the explicit form of $\tf_i{\bf M}^{w_0}$.
We shall recall his result under the identification $\Gamma_n\cong
\cM_n^{\times}$.
\begin{thm}[\cite{Kam2}]\label{thm:AM}
For each $i\in I$, we have
$$(\tf_i{\bf M}^{w_0})_K=\left\{\begin{array}{ll}
\mbox{\em min}\left\{M^{w_0}_K,~M^{w_0}_{s_iK}+c_i({\bf M}^{w_0})\right\} &
(K\in \cM_n^{\times}(i)),\\
M^{w_0}_K & (\mbox{\em otherwise}).
\end{array}\right.$$
Here $c_i({\bf M}^{w_0})=M_{[1,i]}^{w_0}-M^{w_0}_{[1,i+1]\setminus\{i\}}-1$.
\end{thm}
\begin{rem}{\rm
(1) If $K=[1,i]$, then we have $(\tf_i{\bf M}^{w_0})_{[1,i]}=M^{w_0}_{[1,i]}-1$. 
Indeed, $[1,i]$ is an element of $\cM_n^{\times}(i)$. Since 
$s_i[1,i]=[1,i+1]\setminus\{i\}$, we have
\begin{align*}
(\tf_i{\bf M}^{w_0})_{[1,i]}&=
\mbox{min}\left\{M^{w_0}_{[1,i]},~M^{w_0}_{[1,i+1]\setminus\{i\}}+
M_{[1,i]}^{w_0}-M^{w_0}_{[1,i+1]\setminus\{i\}}-1\right\}\\
&=\mbox{min}\left\{M^{w_0}_{[1,i]},~M^{w_0}_{[1,i]}-1\right\}\\
&=M^{w_0}_{[1,i]}-1.
\end{align*}
(2) As we already mentioned in the introduction, the above formula is 
conjectured by Anderson and Mirkovi\'c (unpublished)
(See \cite{Kam2}).
So it is called the {\it Anderson-Mirkovi\'c} ({\it AM for short}) 
{\it conjecture}.
}\end{rem}
By using the above formula, we can also calculate the explicit form of
the $\tf_i^*$-action on an $e$-BZ datum.
\begin{cor}\label{cor:AM-e}
For ${\bf M}^e=(M_K^e)\in \mathcal{BZ}^e$ we have 
$$(\tf_i^*{\bf M}^{e})_K=\left\{\begin{array}{ll}
\mbox{\em min}\left\{M^{e}_K,~M^{e}_{s_iK}+c_i^*({\bf M}^{e})\right\} &
(K\in \cM_n^{\times}(i)^*),\\
M^{e}_K & (\mbox{\em otherwise}).
\end{array}\right.$$
Here $c_i^*({\bf M}^{e})=M_{[1,i]^c}^{e}-M^{e}_{([1,i+1]\setminus\{i\})^c}-1
=M_{[i+1,n+1]}^{e}-M^{e}_{\{i\}\cup[i+2,n+1]}-1$.
\end{cor}

\subsection{Comparison 
}
As we explained above both $\cB$ (with the $\ast$-crystal structure) and 
$\mathcal{BZ}^{e}$ are crystal which are isomorphic to $B(\infty)$. Therefore,
as abstract crystals, they are isomorphic. In this subsection we will construct 
an explicit isomorphism form $\cB$ to $\mathcal{BZ}^{e}$.\\

Following \cite{BFZ}, we introduce a notion of $K$-tableau for a Maya diagram
$K\in\cM_n^{\times}$.
\begin{defn}
Let $K=\{k_1<k_2<\cdots<k_l\}\in\cM_n^{\times}$ be a Maya diagram. For such $K$,
we define a $K$-tableau as an upper-triangular matrix 
$C=(c_{p,q})_{1\leq p\leq q\leq l}$ with integer entries satisfying
$$c_{p,p}=k_p\qquad (1\leq p\leq l),$$
and the usual monotonicity conditions for semi-standard tableaux:
$$c_{p,q}\leq c_{p,q+1},\qquad c_{p,q}<c_{p+1,q}.$$
\end{defn}
For a giving ${\bf i}_0$-Lusztig datum ${\bf a}=(a_{i,j})\in\cB$, let ${\bf M}
({\bf a})=(M_K({\bf a}))_{K\in \cM_n^{\times}}$ be a collection of integers
defined by
$$M_K({\bf a})=-\sum_{j=1}^l\sum_{i=1}^{k_j-1}a_{i,k_j}+
\mbox{min}\left\{\left.\sum_{1\leq p<q\leq l}a_{c_{p,q},c_{p,q}+(q-p)}
~\right|~C=(c_{p,q})\mbox{ is a $K$-tableau }\right\}$$
and denote the map ${\bf a}\mapsto {\bf M}({\bf a})$ by $\Psi$.
\begin{prop}[\cite{BFZ}]
For any ${\bf a}\in\cB$, $\Psi({\bf a})={\bf M}({\bf a})$ is an $e$-BZ datum.
Moreover $\Psi:\cB\to \mathcal{BZ}^{e}$ is a bijection.
\end{prop}
In this article, we prove the next theorem.
\begin{thm}\label{prop:BZ}
The bijection $\Psi:\cB\to \mathcal{BZ}^{e}$ is an isomorphism of crystals
with respect to $\ast$-crystal structures. 
\end{thm}
To prove this theorem, it is enough to show the next two lemmas.
\begin{lemma}\label{lemma:isom1}
For any ${\bf a}\in\cB$, we have 
$$\mbox{\rm wt}({\bf M}({\bf a}))=\mbox{\rm wt}({\bf a}),\quad
\eps_i^*({\bf M}({\bf a}))=\eps_i^*({\bf a}),\quad
\vphi_i^*({\bf M}({\bf a}))=\vphi_i^*({\bf a}).$$
\end{lemma}
\begin{lemma}\label{lemma:isom2}
For any ${\bf a}\in\cB$, we have 
$$\te_i^*({\bf M}({\bf a}))={\bf M}(\te_i^*{\bf a}),\quad
\tf_i^*({\bf M}({\bf a}))={\bf M}(\tf_i^*{\bf a}).$$
Here we set ${\bf M}(0)=0$.
\end{lemma}
\vskip 1mm
\noindent
{\it Proof of Lemma \ref{lemma:isom1}}.
%
Firstly let us compute $\mbox{wt}({\bf M}({\bf a}))$. 
Since ${\bf M}({\bf a})$ is an $e$-BZ datum we have
\begin{align*}
\mbox{wt}({\bf M}({\bf a}))&
=\sum_{i\in I}M_{[1,i]^c}({\bf a})\alpha_i\\
&=\sum_{i\in I}M_{[i+1,n+1]}({\bf a})\alpha_i.
\end{align*}
For $K=[i+1,n+1]$, there exist a unique $K$-tableau 
$$C=(c_{p,q})_{1\leq p\leq q\leq n+1-i}=\left(\begin{array}{cccccccc}
i+1 & i+1 & \cdots & \cdots & i+1\\
      & i+2 & i+2  & \cdots & i+2\\
      &       & \ddots & \ddots & \vdots\\
      &       &        & n      & n\\
      &       &        &        & n+1
\end{array}\right).$$
That is, $c_{p,q}=i+p~(1\leq p\leq q\leq n+1-i)$. Therefore, for any $i\in I$, 
we have
\begin{align*}
M_{[i+1,n+1]}&=-\sum_{t=i+1}^{n+1}\sum_{s=1}^{t-1}a_{s,t}
+\sum_{1\leq p<q\leq n+1-i} a_{i+p,i+p+(q-p)}\\
&=-\sum_{t=i+1}^{n+1}\sum_{s=1}^{i}a_{s,t}\\
&=-m_i.
\end{align*}
Here we set $m_0=m_{n+1}=0$.
This equalities says that $\mbox{\rm wt}({\bf M}({\bf a}))=
\mbox{\rm wt}({\bf a})$.\\

Nextly let us calculate $\eps_i^*({\bf M}({\bf a}))$. We have
\begin{align*}
\eps_i^*({\bf M}({\bf a}))&=\eps_i({\bf M}({\bf a})^*)\\
&=-M_{[i+1,n+1]}({\bf a})-M_{\{i\}\cup [i+2,n+1]}({\bf a})
+M_{[i+2,n+1]}({\bf a})+M_{[i,n+1]}({\bf a}).
\end{align*}
From the proof of the first formula we already know
$$M_{[k+1,n+1]}({\bf a})=-\sum_{t=k+1}^{n+1}\sum_{s=1}^{k}a_{s,t}
\quad(k=i-1,i,i+1).$$
For $K=\{i\}\cup [i+2,n+1]$, the set of all 
$K$-tableaux is given by
$\left\{C^{(r)}\right\}_{1\leq r\leq n+1-i}$
where
$$C^{(r)}=\left(c_{p,q}^{(r)}\right)_{1\leq p\leq q\leq n+1-i}
=\left(\begin{array}{cccccccc}
i & c_{1,2}^{(r)} & \cdots & \cdots & c_{1,n+1-i}^{(r)}\\
      & i+2 & i+2  & \cdots & i+2\\
      &       & \ddots & \ddots & \vdots\\
      &       &        & n      & n\\
      &       &        &        & n+1
\end{array}\right)$$
and 
$$c_{1,q}^{(r)}=\left\{\begin{array}{ll}
i & (2\leq q\leq r),\\
i+1 & (r<q\leq n+1-i).
\end{array}\right.$$
Since 
$$\sum_{q=2}^{n+1-i}a_{c_{1,q}^{(r)},c_{1,q}^{(r)}+(q-1)}=
\sum_{q=i+1}^{i+r-1}a_{i,q}+\sum_{q=i+r+1}^{n+1}a_{i+1,q},$$
we have
\begin{align*}
M_{\{i\}\cup [i+2,n+1]}({\bf a})&=-\sum_{s=1}^{i-1}a_{s,i}
-\sum_{t=i+2}^{n+1}\sum_{s=1}^{t-1}a_{s,t}+\mathop{\mbox{min}}_{1\leq r\leq n+1-i}
\left\{\sum_{1\leq p<q\leq l}a_{c_{p,q}^{(r)},c_{p,q}^{(r)}+(q-p)}\right\}\\
&=-\sum_{s=1}^{i-1}a_{s,i}
-\sum_{t=i+2}^{n+1}\sum_{s=1}^{t-1}a_{s,t}+\sum_{q=i+3}^{n+1}\sum_{p=i+2}^{q-1}a_{p,q}\\
&\qquad\qquad+ 
\mathop{\mbox{min}}_{1\leq r\leq n+1-i}
\left\{\sum_{q=2}^{n+1-i}a_{c_{1,q}^{(r)},c_{1,q}^{(r)}+(q-1)}\right\}\\
&=-\sum_{s=1}^{i-1}a_{s,i}-\sum_{t=i+2}^{n+1}\sum_{s=1}^{i+1}a_{s,t}
+\mathop{\mbox{min}}_{1\leq r\leq n+1-i}
\left\{\sum_{q=i+1}^{i+r-1}a_{i,q}+\sum_{q=i+r+1}^{n+1}a_{i+1,q}\right\}.
\end{align*}
Putting together all formulas, we have
\begin{align*}
\eps_i^*({\bf M}({\bf a}))&=\sum_{t=i+1}^{n+1}\sum_{s=1}^{i}a_{s,t}-
\sum_{t=i}^{n+1}\sum_{s=1}^{i-1}a_{s,t}-\sum_{t=i+2}^{n+1}\sum_{s=1}^{i+1}a_{s,t}
+\sum_{s=1}^{i-1}a_{s,i}+
\sum_{t=i+2}^{n+1}\sum_{s=1}^{i+1}a_{s,t}\\
&\qquad\qquad
-\mathop{\mbox{min}}_{1\leq r\leq n+1-i}
\left\{\sum_{q=i+1}^{i+r-1}a_{i,q}+\sum_{q=i+r+1}^{n+1}a_{i+1,q}\right\}\\
&=\sum_{t=i+1}^{n+1}a_{i,t}-\mathop{\mbox{min}}_{1\leq r\leq n+1-i}
\left\{\sum_{q=i+1}^{i+r-1}a_{i,q}+\sum_{q=i+r+1}^{n+1}a_{i+1,q}\right\}\\
&=\mathop{\mbox{max}}_{1\leq r\leq n+1-i}
\left\{\sum_{q=i+r}^{n+1}a_{i,q}-\sum_{q=i+r+1}^{n+1}a_{i+1,q}\right\}\\
&=\mathop{\mbox{max}}_{1\leq r\leq n+1-i}
\left\{A_{i-1+r}^{\ast(i)}({\bf a})\right\}\\
&=\eps_i^*({\bf a}).
\end{align*}
Finally let us prove $\vphi_i^*({\bf M}({\bf a}))=\vphi_i^*({\bf a})$. But it
is clear by the first and second formulas.
\hspace*{10mm}\hfill $\square$
\vskip 3mm
We can prove Lemma \ref{lemma:isom2} by direct calculation. But in this article
we give another proof by using a Lagrangian constriction of $B(\infty)$, which
we will explain later. (See Subsection 7.3.)
\section{Quivers of type $A_n$}
\subsection{Quivers and their representations} 
Let $(I,H)$ be the double quiver of type $A_n$. Here $I=\{1,2,\cdots,n\}$ is
the set of vertices and $H$ is the set of arrows. If $\tau\in H$ is the arrow
from $i$ to $j$, we denote $\out(\tau)=i$ and $\inn(\tau)=j$. For that 
$\tau\in H$, let $\overline{\tau}$ be the arrow from $j$ to $i$. The map
$\tau\mapsto \overline{\tau}$ defines an involution of $H$. An orientation
$\Omega$ is a subset of $H$ such that $\Omega\cap\overline{\Omega}=\phi$ and
$\Omega\cup\overline{\Omega}=H$. Then $(I,\Omega)$ is a Dynkin quiver of type
$A_n$. 

Let ${\bf V}=(V,B)$ be a representation of the quiver $(I,\Omega)$. Here
$V=\oplus_{i\in I}V_i$ be a finite dimensional $I$-graded complex vector space
with the dimension vector $\dim V=(\dimc V_i)_{i\in I}\in\nz_{\geq 0}^I$, 
and $B=(B_{\tau})_{\tau\in\Omega}$ is a collection of linear maps
$B_{\tau}:V_{\sout(\tau)}\to V_{\sinn(\tau)}$. 
We denote by $M\Omega$ 
the category of representations of the quiver $(I,\Omega)$. Let 
${\bf V}=(V,B)$, ${\bf V}'=(V',B')\in M\Omega$. A morphism 
$\phi=(\phi_i)_{i\in I}$ form ${\bf V}$ to 
${\bf V}'$ is a collection of linear maps $\phi_i:V_i\to V_i'$ such that
$\phi_{\sinn(\tau)}B_{\tau}=B_{\tau}'\phi_{\sout(\tau)}$ for any $\tau\in\Omega$. 
It is well-known that $M\Omega$ is a Krull-Schmidt category. That is, any
object in $M\Omega$ has a unique indecomposable decomposition.
For $i\in I$ let ${\bf e}(i;\Omega)$ be a representation of 
$(I,\Omega)$ defined by $V_i=\nc$ and $V_j=0$ for $j\ne i$. 
This is simple and any simple representation isomorphic to ${\bf e}(i;\Omega)$,
for a unique $i$. 

Assume that $i\in I$ is a sink ({\it resp.} a source) of an orientation 
$\Omega$. That is, there is no arrow $\tau\in\Omega$ such that $\out(\tau)=i$
({\it resp.} $\inn(\tau)=i$). We denote by $\mbox{sink}(\Omega)$ ({\it resp.}
$\mbox{source}(\Omega)$) the set of all sink ({\it resp.} source) vertices. 
Let $s_i\Omega$ be the orientation obtained from
$\Omega$ by reversing each arrow $\tau$ such that $\inn(\tau)=i$ ({\it resp.}
$\out(\tau)=i$).   
\begin{defn}
Fix an orientation $\Omega$. A reduced word ${\bf i}=(i_1,\cdots,i_N)$ of $w_0$
is said to be adapted to $\Omega$ if $i_k$ is a sink of 
$\Omega_k=s_{i_{k-1}}\cdots s_1\Omega$ for $1\leq k\leq N$.
\end{defn}

For a representation of a quiver ${\bf V}=(V,B)$, we set $\dim {\bf V}=\dim V$.
From now on, we identify the dimension vector 
$\dim V=(\dimc V_i)_{i\in I}\in\nz_{\geq 0}^I$ with an element of 
$Q_+=\oplus_{i\in I}\nz_{\geq 0}\alpha_i$ by
$$(\dimc V_i)_{i\in I}\mapsto \sum_{i\in I}(\dimc V_i)\alpha_i.$$

\begin{prop}{\cite{L1}}
{\rm (1)} For a giving orientation $\Omega$, there exist a reduced word 
${\bf i}$ of $w_0$ adapted to $\Omega$.
\vskip 1mm
\noindent
{\rm (2)} For each $\beta\in \Delta^+$, there is a 
unique indecomposable representation {\rm ({\em up to isomorphism})} 
${\bf e}({\beta};\Omega)$ such that $\dim {\bf e}({\beta};\Omega)=\beta$. 
Moreover any indecomposable representation is isomorphic to 
${\bf e}({\beta};\Omega)$ for a unique $\beta$ {\rm ({\em Gabriel's theorem})}. 
\vskip 1mm
\noindent
{\rm (3)} If $\beta >_{~{\bf i}}\beta'$, we have
$\mbox{\rm Hom}_{M\Omega}({\bf e}({\beta};\Omega),{\bf e}({\beta'};\Omega))=0$.
\end{prop}
\subsection{Orientations arising from Maya diagrams}
Any Maya diagram
$K\in\cM_n^{\times}$ can be written as a disjoint union of intervals
$$\begin{array}{c}
K=[s_1+1,t_1]\sqcup [s_2+1,t_2]\sqcup\cdots\sqcup[s_l+1,t_l]\\
(0\leq s_1<t_1<s_2<t_2<\cdots <s_l<t_l\leq n+1);\end{array}$$
the interval $K_m=[s_m+1,t_m]~(1\leq m\leq l)$ will be called the 
$m$-th component of $K$. Define two subsets $\out(K)$ and $\inn(K)$ of 
$[1,n]$ by
$$\out(K)=\{t_m|~1\leq m\leq l\}\cap [1,n],
\quad \inn(K)=\{s_m|~1\leq m\leq l\}\cap [1,n].$$
We remark that $\out(K)\cap\inn(K)=\phi$. 
Introduce two subsets $I_t$ and $I_s$ as follows:
$$\begin{array}{lll}
I_t&=&\left\{\begin{array}{ll}
\out(K)\cup\{1,n\} & (s_1\geq 2,~t_l=n+1),\\
\out(K)\cup\{1\} & (s_1\geq 2,~t_l\leq n),\\
\out(K)\cup\{n\} & (s_1\leq 1,~t_l=n+1),\\
\out(K) & (s_1\leq 1,~t\leq n).
\end{array}\right.\\
& &\\
I_s&=&\left\{\begin{array}{ll}
\inn(K)\cup\{1,n\} & (s_1=0,~t_l\leq n-1),\\
\inn(K)\cup\{1\} & (s_1=0,~t_l\geq n),\\
\inn(K)\cup\{n\} & (s_1\geq 1,~t_l\leq n-1),\\
\inn(K) & (s_1\geq 1,~t_l\geq n).
\end{array}\right.
\end{array}$$
\begin{defn}
{\rm (1)} 
In the above setting, there exist a unique orientation $\Omega(K)$ so that
$\mbox{\em source}(\Omega(K))=I_t$ and $\mbox{\em sink}(\Omega(K))=I_s$. 
We call $\Omega(K)$ the orientation arising from a Maya diagram 
$K\in\cM^{\times}_n$.
\vskip 1mm
\noindent
{\rm (2)} Let $s_K=\mbox{\em min}\{k~|~k\not\in K\}$ and
$t_K=\mbox{\em max}\{k~|~k\in K\}$. Define 
$\beta_K\in \Delta^+\cup\{0\}$
by 
$$\beta_K=\left\{\begin{array}{ll}
\alpha_{s_K}+\alpha_{s_K+1}+\cdots+\alpha_{t_K-1} & (s_K<t_K),\\
0 & (\mbox{otherwise})
\end{array}\right.$$
and we call it the characterizing positive root of a Maya diagram $K$.
\end{defn}
\begin{rem}{\rm (1) In general, we have
$$\mbox{out}(K)\subset \mbox{source}(\Omega(K)),\qquad
\mbox{in}(K)\subset \mbox{sink}(\Omega(K)).$$
(2) The characterizing positive root $\beta_K=0$ if and only if $K=[1,t_1]$ for
some $1\leq t_1\leq n$.
}\end{rem}
\begin{ex}{\rm
Let $n=17$ and $K=[3,4]\sqcup [7,8]\sqcup [10,12]\sqcup [14,15]$. Then we have
$$\mbox{out}(K)=\{4,8,12,15\},\qquad \mbox{in}(K)=\{2,6,9,13\}.$$
Since $s_1=2$ and $t_l=t_4=15$, we have
$$I_t=\out(K)\cup\{1\}=\{1,4,8,12,15\},\qquad  
I_s=\inn(K)\cup\{17\}=\{2,6,9,13,17\}.$$
In this case, the orientation $\Omega(K)$ is given as follows:
\begin{center}
\setlength{\unitlength}{1mm}
\begin{picture}(123,10)
\put(0,5){$\Omega(K)=$}
\put(17,6){\vector(1,0){4}}\put(27,6){\vector(-1,0){4}}
\put(33,6){\vector(-1,0){4}}\put(35,6){\vector(1,0){4}}
\put(41,6){\vector(1,0){4}}\put(51,6){\vector(-1,0){4}}
\put(57,6){\vector(-1,0){4}}\put(59,6){\vector(1,0){4}}
\put(69,6){\vector(-1,0){4}}\put(75,6){\vector(-1,0){4}}
\put(81,6){\vector(-1,0){4}}\put(83,6){\vector(1,0){4}}
\put(93,6){\vector(-1,0){4}}\put(99,6){\vector(-1,0){4}}
\put(101,6){\vector(1,0){4}}\put(107,6){\vector(1,0){4}}
\put(15,1){{\scriptsize $1$}}
\put(21,1){{\scriptsize $2$}}
\put(33,1){{\scriptsize $4$}}
\put(45.5,1){{\scriptsize $6$}}
\put(57.5,1){{\scriptsize $8$}}
\put(63.5,1){{\scriptsize $9$}}
\put(80.5,1){{\scriptsize $12$}}
\put(86.5,1){{\scriptsize $13$}}
\put(98.5,1){{\scriptsize $15$}}
\put(110.5,1){{\scriptsize $17$}}
\put(16,6){\circle*{2}}\put(22,6){\circle{2}}\put(27.5,5){$\cdot$}
\put(34,6){\circle*{2}}\put(39.5,5){$\cdot$}
\put(46,6){\circle{2}}\put(51.5,5){$\cdot$}
\put(58,6){\circle*{2}}\put(64,6){\circle{2}}
\put(69.5,5){$\cdot$}\put(75.5,5){$\cdot$}\put(82,6){\circle*{2}}
\put(88,6){\circle{2}}\put(93.5,5){$\cdot$}\put(100,6){\circle*{2}}
\put(105.5,5){$\cdot$}\put(112,6){\circle{2}}
\put(115,5){.}
\end{picture}
\end{center}
Here $\circ$ is a sink and $\bullet$ is a source．

Since $s_K=1$ and $t_K=15$, the characterizing positive root $\beta_K$ is 
given by
$$\beta_K=\sum_{i=1}^{14}\alpha_i.$$
}\end{ex}
\subsection{From Lusztig data to $e$-BZ data}
Let ${\bf i}$ be a reduced word adapted to the orientation $\Omega(K)$ and
consider the set of all ${\bf i}$-Lusztig data 
$$\cB^{\bf i}
=\left\{\left.{\bf a}^{\bf i}=(a^{\bf i}_{i,j})_{(i,j)\in \Pi}~\right|~
a^{\bf i}_{i,j}\in\nz_{\geq 0}\right\}.$$ 
Recall the identification 
$\Delta^+\overset{\sim}{\to}\Pi$ (see Subsection 2.4) and denote the image of
$\beta\in \Delta^+$ by $(i_{\beta},j_{\beta})\in \Pi$. Set
${\bf e}((i_{\beta},j_{\beta});\Omega(K))={\bf e}({\beta};\Omega(K))$. 
Then, for each ${\bf V}\in M\Omega(K)$, there is a unique 
${\bf a}^{\bf i}\in\cB^{\bf i}$ such that ${\bf V}$ is isomorphic to 
${\bf V}({\bf a}^{\bf i})$. Here
$${\bf V}({\bf a}^{\bf i})=
\mathop{\oplus}_{(i,j)\in \Pi}{\bf e}((i,j);\Omega(K))^{\oplus a_{i,j}^{\bf i}}.$$
We introduce a non-positive integer 
$$M_K({\bf V}({\bf a}^{\bf i}))=-\dimc\mbox{Hom}_{M\Omega(K)}
\left({\bf V}({\bf a}^{\bf i}),{\bf e}(\beta_K;\Omega(K))\right).$$
\begin{lemma}\label{lemma:quiver-BZ}
{\rm (1)} We have
$$M_K({\bf V}({\bf a}^{\bf i}))=-\sum_{(i,j)\in\Pi;i\not\in K,j\in K}a_{i,j}^{\bf i}.
$$
\noindent
{\rm (2)} Denote ${\bf V}({\bf a}^{\bf i})=
(\oplus_iV_i,(B_{\tau})_{\tau\in\Omega(K)})$. Then we have
$$\sum_{(i,j)\in\Pi;i\not\in K,j\in K}a_{i,j}^{\bf i}=
\dimc\mbox{\rm Coker}\left(\mathop{\oplus}_{k\in \mbox{\rm out}(K)}V_k
\overset{\oplus B_{\sigma}}{\longrightarrow}
\mathop{\oplus}_{l\in \mbox{\rm in}(K)}V_l\right).$$
Here $\sigma$ is a path in $\Omega(K)$ form some $k\in \mbox{\rm out}(K)$ to 
some $l\in\mbox{\rm in}(K)$
\end{lemma}
The proof of this lemma will be given in the next subsection. \\

The next proposition is a easy consequence of the results of 
Berenstein, Fomin and Zelevinsky \cite{BFZ}. 
\begin{prop}\label{prop:BFZ}
Let ${\bf i}$ be a reduced word adapted to the orientation $\Omega(K)$
and ${\bf a}\in \cB$ an ${\bf i}_0$-Lusztig datum. Set 
${\bf a}^{\bf i}=R_{{\bf i}_0}^{\bf i}({\bf a})$. 
Here $R_{{\bf i}_0}^{\bf i}$ is the transition map
from ${\bf i}_0$ to ${\bf i}$. Then we have
$$M_{K}({\bf a})=-\sum_{(i,j)\in\Pi;i\not\in K,j\in K}a_{i,j}^{\bf i}.$$
Here ${\bf M}({\bf a})=(M_K({\bf a}))_{K\in\cM_n^{\times}}$ is the
$e$-BZ datum defined in 5.6. 
\end{prop}
Combining the above results, we have the following corollary:
\begin{cor}\label{cor:BZ}
In the above setting, we have
$$M_{K}({\bf a})=M_K({\bf V}({\bf a}^{\bf i}))=
-\dimc\mbox{\rm Coker}\left(\mathop{\oplus}_{k\in \mbox{\rm out}(K)}V_k
\overset{\oplus B_{\sigma}}{\longrightarrow}
\mathop{\oplus}_{l\in \mbox{\rm in}(K)}V_l\right).$$
\end{cor}
\subsection{Proof of Lemma \ref{lemma:quiver-BZ}}
Let us prove the formula (1). It is enough to show that
$$\dimc\mbox{Hom}_{M\Omega(K)}
\left({\bf e}((i,j);\Omega(K)),{\bf e}(\beta_K;\Omega(K))\right)=
\left\{\begin{array}{ll}
1 & (i\not\in K\mbox{ and }j\in K),\\
0 & (\mbox{otherwise}).
\end{array}\right.\eqno{(6.4.1)}$$
Let us denote ${\bf e}((i,j);\Omega(K))=
(\oplus V_k',(B_{\tau}'))$ and ${\bf e}(\beta_K;\Omega(K))
=(\oplus V_k'',(B_{\tau}''))$. Then
$$V_k'=\left\{\begin{array}{ll}
\nc & (i\leq k\leq j-1),\\
0 & (\mbox{otherwise})
\end{array}\right.\quad\mbox{and}\quad
V_k''=\left\{\begin{array}{ll}
\nc & (s_K\leq k\leq t_K-1),\\
0 & (\mbox{otherwise}).
\end{array}\right.
$$

Firstly assume $i\not\in K\mbox{ and }j\in K$. Since our quiver is of 
type $A_n$, the left hand side of $(6.4.1)$ is less than $1$. So it is enough 
to show that there is a non-trivial morphism form 
$(\oplus V_k',(B_{\tau}'))$ to $(\oplus V_k'',(B_{\tau}''))$. 
By the assumption we have $s_K\leq i<j\leq t_K$. 
Therefore we can define 
a linear map $\psi=(\psi_k):\oplus V_k'\to
\oplus V_k''$ by
$$\psi_k=\left\{\begin{array}{ll}
\mbox{id}_{\snc} & (i\leq k\leq j-1)\\
0 & (\mbox{otherwise})
\end{array}\right.$$
and it is easy to check that the above map is a non-trivial 
morphism of $M\Omega(K)$. 

Secondly let us consider the case that $i\in K\mbox{ or }j\not\in K$.
The goal is to prove 
$$\mbox{Hom}_{M\Omega(K)}\left((\oplus V_k',(B_{\tau}')),
(\oplus V_k'',(B_{\tau}''))\right)=0.\eqno{(6.4.2)}$$

Assume $i\in K$ and $s_K<i$. Then $i\geq 2$, $i+1\leq t_K$ and there is an
arrow $\tau_1$ form $i$ to $i-1$ in $\Omega(K)$. 
Let $\psi=(\psi_k)\in \mbox{Hom}_{M\Omega(K)}\left((\oplus V_k',(B_{\tau}')),
(\oplus V_k'',(B_{\tau}''))\right)$. 
Since $V_{i-1}'=0$, we have $B''_{\tau_1}\psi_i=\psi_{i-1}B_{\tau_1}'=0$.
On the other hand, $B''_{\tau_1}\ne 0$ because $s_K\leq i-1$.
Therefore we have $\psi_i=0$.

If $j=i+1$ or $t_K=i+1$, it means that the left hand side of 
$(6.4.2)$ is equal to zero. So we may assume $j>i+1$ and $t_K>i+1$. 
However one can show that $\psi_{i+1}=0$. Indeed, if there is an arrow 
$\tau_2$ form $i$ to $i+1$ in $\Omega(K)$,
then we have $\psi_{i+1}B_{\tau_2}'=B_{\tau_2}''\psi_i$ with non-trivial 
$B_{\tau_2}'$
and $B_{\tau_2}''$. Since $\psi_i=0$, we have $\psi_{i+1}=0$. On the other hand,
if there is an arrow $\tau_2$ form $i+1$ to $i$ in $\Omega(K)$, then we 
have $\psi_iB_{\tau_2}'=B_{\tau_2}''\psi_{i+1}$. Then we also have $\psi_{i+1}=0$.
By repeating this method, we have $\psi_k=0$ for any $k\in I$. 

For the other cases, we can show (6.4.2)
by similar way.\\

We will give a proof of (2). Since ${\bf V}({\bf a}^{\bf i})=
\oplus_{(i,j)\in \Pi}{\bf e}((i,j);\Omega(K))^{\oplus a_{i,j}^{\bf i}}$ is the 
indecomposable decomposition, it is enough to prove that
$$\dimc\mbox{\rm Coker}\left(\mathop{\oplus}_{k\in \mbox{\rm out}(K)}V_k'
\overset{\oplus B_{\sigma}'}{\longrightarrow}
\mathop{\oplus}_{l\in \mbox{\rm in}(K)}V_l'\right)=\left\{
\begin{array}{ll}
1 & (i\not\in K\mbox{ and }j\in K),\\
0 & (\mbox{otherwise}).
\end{array}
\right.\eqno{(6.4.3)}$$

Recall the decomposition of $K$:
$$K=K_1\sqcup\cdots\sqcup K_l,\quad \mbox{where }
K_m=[s_m+1,t_m]\quad (1\leq m\leq l).$$
 
Firstly assume $i\not\in K\mbox{ and }j\in K$. More precisely, we assume 
$t_{u-1}<i<s_u+1$ and $s_v+1\leq j\leq t_v$ with $u\leq v$. Let 
$\sigma(t_{m-1}\to s_{m})$ ({\it resp}. $\sigma(s_m\leftarrow t_m)$) be the path 
form $t_{m-1}$ to $s_{m}$ ({\it resp}. from $t_m$ to $s_m$) in $\Omega(K)$. 
Then, by the definition, we have 
$$B'_{\sigma(t_{m-1}\to s_{m})}=\left\{\begin{array}{ll}
\mbox{id}_{\snc} & (u+1\leq m\leq v),\\
0 & (\mbox{otherwise}),
\end{array}\right.\quad
B'_{\sigma(s_m\leftarrow t_m)}=\left\{\begin{array}{ll}
\mbox{id}_{\snc} & (u\leq m\leq v-1),\\
0 & (\mbox{otherwise}).
\end{array}\right.$$ 
Therefore we have
\begin{align*}
\dimc\mbox{\rm Coker}\left(\mathop{\oplus}_{k\in \mbox{\rm out}(K)}V_k'
\overset{\oplus B_{\sigma}'}{\longrightarrow}
\mathop{\oplus}_{l\in \mbox{\rm in}(K)}V_l'\right)&=
\dimc\mbox{\rm Coker}\left(\mathop{\oplus}_{m=u}^{v-1}V_{t_m}'
{\longrightarrow}
\mathop{\oplus}_{m=u}^{v}V_{s_m}'\right)\\
&=\dimc\mbox{\rm Coker}\left(\nc^{v-u-1}\hookrightarrow \nc^{v-u}\right)\\
&= 1.
\end{align*}

Secondly assume $i,j\in K$. Then there exist $u$ and $v$ with $u\leq v$ such
that $i\in K_u$ and $j\in K_v$. In this case, we have
$$B'_{\sigma(t_{m-1}\to s_{m})}=\left\{\begin{array}{ll}
\mbox{id}_{\snc} & (u+1\leq m\leq v),\\
0 & (\mbox{otherwise}),
\end{array}\right.\quad
B'_{\sigma(s_m\leftarrow t_m)}=\left\{\begin{array}{ll}
\mbox{id}_{\snc} & (u+1\leq m\leq v-1),\\
0 & (\mbox{otherwise}).
\end{array}\right.$$ 
Therefore we have
\begin{align*}
\dimc\mbox{\rm Coker}\left(\mathop{\oplus}_{k\in \mbox{\rm out}(K)}V_k'
\overset{\oplus B_{\sigma}'}{\longrightarrow}
\mathop{\oplus}_{l\in \mbox{\rm in}(K)}V_l'\right)&=
\dimc\mbox{\rm Coker}\left(\mathop{\oplus}_{m=u}^{v-1}V_{t_m}'
{\longrightarrow}
\mathop{\oplus}_{m=u+1}^{v}V_{s_m}'\right)\\
&=\dimc\mbox{\rm Coker}\left(\nc^{v-u-1}\overset{\sim}{\to} \nc^{v-u-1}\right)\\
&= 0.
\end{align*}
For the other cases, we can prove that the left hand side of (6.4.3) equals
to zero by similar arguments. Thus, the lemma is proved.
\section{Lagrangian construction of crystal basis}
\subsection{Varieties associated to quivers}
For $\nu\in Q_+$, let $\cV_{\nu}$ be the category of $I$-graded complex 
vector spaces $V$ with $\dim V=\nu$. 
For $V=\oplus_{i\in I}V_i\in \cV_{\nu}$, introduce two complex vector spaces
$$E_{V,\Omega}=\mathop{\oplus}_{\tau\in\Omega}\homc(V_{\sout(\tau)},
V_{\sinn(\tau)}),\qquad 
X_{V}=\mathop{\oplus}_{\tau\in H}\homc(V_{\sout(\tau)},V_{\sinn(\tau)}).$$
An element of $E_{V,\Omega}$ or $X_{V}$ will be denoted by 
$B=(B_{\tau})$ where 
$B_{\tau}\in \homc(V_{\sout(\tau)},V_{\sinn(\tau)})$. Define a 
symplectic form 
$\omega$ on $X_{V}$ by
$$\omega(B,B')=\sum_{\tau\in H}\eps(\tau)\mbox{tr}(B_{\overline{\tau}}B_{\tau}')$$
where $\eps(\tau)=1$ for $\tau\in\Omega$ and $\eps(\tau)=-1$ for $\tau\in
\overline{\Omega}$. We regard $X_{V}$ as the cotangent bundle 
$T^*E_{V,\Omega}$ of $E_{V,\Omega}$ via the symplectic form $\omega$.  

The group $G_{V}=\prod_{i\in I}
GL(V_i)$ acts on $E_{V,\Omega}$ and $X_{V}$ by
$$G_V\ni g=(g_i):(B_{\tau})\mapsto (g_{\sinn(\tau)}B_{\tau}g_{\sout(\tau)}^{-1}).$$
Since the action of $G_{V}$ on $X_V$ preserves the symplectic form $\omega$, 
we can consider the corresponding moment map
$\mu:X_V\to \bigl(\gtg_V\bigr)^*\cong \gtg_V$. Here $\gtg_V=\mbox{Lie }G_V$
and we identify $\gtg_V$ with its dual via the Killing form. Set
$$\Lambda_V=\mu^{-1}(0).$$
It is known that $\Lambda_V$ is a $G_V$-invariant closed Lagrangian subvariety
of $X_V$. It is clear that, for $V,~V'\in\cV_{\nu}$, there are natural 
isomorphisms $V\cong V'$, $X_V\cong X_{V'}$ and $\Lambda_V\cong \Lambda_{V'}$. 
Hence we denote them $V(\nu)$, $X(\nu)$ and $\Lambda(\nu)$, respectively.

Let $\mbox{Irr}\Lambda(\nu)$ be a set of all irreducible components of 
$\Lambda(\nu)$.
Since our quiver is of type $A_n$, 
there is a bijection from the set of all $G_{V(\nu)}$-orbits in 
$E_{V(\nu),\Omega}$ to $\mbox{Irr}\Lambda(\nu)$ defined by $\mathcal{O}\mapsto
\overline{T_{\mathcal{O}}^*E_{V(\nu),\Omega}}$. We remark that $E_{V(\nu),\Omega}$ has
finitely many $G_{V(\nu)}$-orbits because our quiver is of type $A_n$.\\
 
For $B\in E_{V(\nu),\Omega}$, a pair ${\bf V}=(V(\nu),B)$ is nothing but a 
representation of a quiver $(I,\Omega)$ with a dimension vector $\nu$.
Moreover there is a natural one to one correspondence
between isomorphism classes of representations of a quiver $(I,\Omega)$ with
a dimension vector $\nu$ and $G_{V(\nu)}$-orbits in $E_{V(\nu),\Omega}$.

Let $\Omega$ be an orientation and ${\bf i}$ a reduced word adapted to 
$\Omega$. As we mentioned before, for each ${\bf V}=(V(\nu),B)
\in M\Omega$, there is a 
unique ${\bf i}$-Lusztig datum ${\bf a}^{\bf i}\in\cB^{\bf i}$ such that 
${\bf V}$ is isomorphic to 
$${\bf V}({\bf a}^{\bf i})=
\mathop{\oplus}_{(i,j)\in \Pi}{\bf e}((i,j);\Omega)^{\oplus a_{i,j}^{\bf i}}.$$
Let $\cO_{{\bf a}^{\bf i}}$ be the $G_{V(\nu)}$-orbit of $E_{V(\nu),\Omega}$ through 
${\bf V}({\bf a}^{\bf i})$. Denote $\Lambda_{{\bf a}^{\bf i}}=
\overline{T_{\cO_{{\bf a}^{\bf i}}}^*E_{V(\nu),\Omega}}$. Then we have a bijection
$\Psi_{\bf i}:\cB^{\bf i}\overset{\sim}{\to}
\bigsqcup_{\nu\in Q^-}\mbox{Irr}\Lambda(\nu)$ defined by
${\bf a}^{\bf i}\mapsto \Lambda_{{\bf a}^{\bf i}}$.
Especially, consider the following special orientation
\begin{center}
\setlength{\unitlength}{1mm}
\begin{picture}(103,10)
\put(4,5){$\Omega_0:$}
\put(27,6){\vector(-1,0){6}}\put(17,6){\line(1,0){10}}
\put(39,6){\vector(-1,0){6}}\put(29,6){\line(1,0){10}}
\put(51,6){\vector(-1,0){6}}\put(41,6){\line(1,0){10}}
\put(54,4.5){$\cdots$}
\put(74,6){\vector(-1,0){6}}\put(64,6){\line(1,0){10}}
\put(86,6){\vector(-1,0){6}}\put(76,6){\line(1,0){10}}
\put(98,6){\vector(-1,0){6}}\put(88,6){\line(1,0){10}}
\put(15,1){{\scriptsize $1$}}
\put(27.5,1){{\scriptsize $2$}}
\put(39.5,1){{\scriptsize $3$}}
\put(71.5,1){{\scriptsize $n-2$}}
\put(83.5,1){{\scriptsize $n-1$}}
\put(98.5,1){{\scriptsize $n$}}
\put(16,6){\circle{2}}\put(28,6){\circle{2}}
\put(40,6){\circle{2}}\put(75,6){\circle{2}}
\put(87,6){\circle{2}}\put(99,6){\circle{2}}
\put(101,5){.}
\end{picture}
\end{center}
Then the lexicographically minimal reduced word ${\bf i}_0$ is adapted to
$\Omega_0$. For ${\bf a}\in\cB$, let $\cO_{\bf a}$ be the corresponding orbit 
in $E_{V(\nu),\Omega_0}$ and $\Lambda_{{\bf a}}=
\overline{T_{\cO_{{\bf a}}}^*E_{V(\nu),\Omega_0}}$. We remark that  
$$\Lambda_{\bf a}=\Lambda_{{\bf a}^{\bf i}},\eqno{(7.1.1)}$$
where ${\bf i}$ is an arbitrarily reduced word and ${\bf a}^{\bf i}=
R_{{\bf i}_0}^{\bf i}({\bf a})$.
\begin{rem}{\rm
It seems to us that the formula (7.1.1) is known for experts. However 
the proof of it was not appeared until recently. A detailed proof was
firstly given by Kimura in his Master thesis \cite{Kim}
(see Appendix A, in detail).
  
On the other hand, in 2010, Baumann and Kamnitzer give another explicit proof 
of it by using representation theory of preprojective algebra (see \cite{BK}).
}\end{rem}

For $B\in X(\nu)$ we set
$$M_K(B)=-\dimc\mbox{\rm Coker}\left(\mathop{\oplus}_{k\in \mbox{\rm out}(K)}
V(\nu)_k\overset{\oplus B_{\sigma}}{\longrightarrow}
\mathop{\oplus}_{l\in \mbox{\rm in}(K)}V(\nu)_l\right)$$
and for $\Lambda\in \mbox{Irr}\Lambda(\nu)$ define
$$M_K(\Lambda)=M_K(B)$$
by taking a generic point $B$ of $\Lambda$. By Corollary \ref{cor:BZ}, we 
immediately have the following statement.
\begin{cor}\label{cor:Lag}
Recall the setting of Corollary \ref{cor:BZ}: let ${\bf i}$ be a reduced word 
which is adapted to the orientation $\Omega(K)$ and 
${\bf a}\in\cB$ an ${\bf i}_0$-Lusztig datum. 
Set ${\bf a}^{\bf i}=R_{{\bf i}_0}^{\bf i}({\bf a})$. Then we have 
$$M_{K}({\bf a})=M_K({\bf V}({\bf a}^{\bf i}))=M_K(\Lambda_{{\bf a}^{\bf i}}).$$
\end{cor}
Combining the above corollary with (7.1.1), we have
$$M_{K}({\bf a})=M_K(\Lambda_{{\bf a}}).\eqno{(7.1.2)}$$
\subsection{Lagrangian construction of $B(\infty)$}
In this subsection we will give a review of Lagrangian construction of 
$B(\infty)$ following \cite{KS}.\\

Let $\nu,\nu',\overline{\nu}\in Q_+$ with $\nu=\nu'+\overline{\nu}$.
Consider a diagram
$$\Lambda(\nu')\times\Lambda(\overline{\nu})\overset{q_1}{\longleftarrow}
\Lambda(\nu',\onu)\overset{q_2}{\longrightarrow}\Lambda(\nu).
\eqno{(7.2.1)}$$
Here $\Lambda(\nu',\onu)$ is a variety of $(B,\phi',\ophi)$, where
$B\in \Lambda(\nu)$ and $\phi'=(\phi_i'),~\ophi=(\ophi_i)$ give an exact 
sequence
$$0\longrightarrow V(\nu')_i\overset{\phi_i'}{\longrightarrow} V(\nu)
\overset{\ophi_i}{\longrightarrow}V(\onu)\longrightarrow 0$$
such that $\mbox{Im }\phi'$ is stable by $B$. Hence $B$ induces 
$B':V(\nu')\to V(\nu')$ and $\overline{B}:V(\onu)\to V(\onu)$. 
The maps $q_1$ and $q_2$ are defined by
$q_1(B,\phi',\ophi)=(B',\overline{B})$ and $q_2(B,\phi',\ophi)=B$, 
respectively.

For $i\in I$ and $\Lambda\in\mbox{Irr}\Lambda(\nu)$, set
$$\eps_i(\Lambda)=\eps_i(B)\quad\mbox{and}\quad \eps_i^*(\Lambda)=\eps_i^*(B),$$
where $B$ is a general point of $\Lambda$ and 
$$\eps_i(B)=\dimc\mbox{\rm Coker}\left(\mathop{\oplus}_{\tau;\sinn(\tau)=i}
V(\nu)_{\sout(\tau)}
\overset{\oplus B_{\tau}}{\longrightarrow}V(\nu)_i\right),$$
$$\eps_i^*(B)=\dimc\mbox{\rm Ker}\left(V(\nu)_i
\overset{\oplus B_{\tau}}{\longrightarrow}
\mathop{\oplus}_{\tau;\sout(\tau)=i}V(\nu)_{\sinn(\tau)}\right).$$
For $k,l\in \nz_{\geq 0}$, we define
$$\bigl(\mbox{Irr}\Lambda(\nu)\bigr)_{i,k}=
\{\Lambda\in \mbox{Irr}\Lambda(\nu)~|~\eps_i(\Lambda)=k\}~~\mbox{ and }~~
\bigl(\mbox{Irr}\Lambda(\nu)\bigr)_{i}^l=
\{\Lambda\in \mbox{Irr}\Lambda(\nu)~|~\eps_i^*(\Lambda)=l\}.$$ 
Assume $\onu=c\alpha_i$ ({\it resp}. $\nu'=c\alpha_i$) for 
$c\in\nz_{\geq 0}$. Since $\Lambda(c\alpha_i)=\{0\}$, we have the following 
diagrams as special cases of (7.2.1): 
$$\Lambda(\nu')\cong\Lambda(\nu')\times\Lambda(c\alpha_i)
\overset{q_1}{\longleftarrow}
\Lambda(\nu',c\alpha_i)\overset{q_2}{\longrightarrow}\Lambda(\nu),
\eqno{(7.2.2)}$$
$$\Lambda(\onu)\cong\Lambda(c\alpha_i)\times\Lambda(\onu)
\overset{q_1}{\longleftarrow}
\Lambda(c\alpha_i,\onu)\overset{q_2}{\longrightarrow}\Lambda(\nu).
\eqno{(7.2.3)}$$
It is known that the diagrams (7.2.2) and (7.2.3) induce bijections 
$$\te_i^{max}:\bigl(\mbox{\rm Irr}\Lambda(\nu)\bigr)_{i,c}
\overset{\sim}{\to}
\bigl(\mbox{\rm Irr}\Lambda(\nu')\bigr)_{i,0}
\quad\mbox{and}\quad
\te_i^{\ast max}:\bigl(\mbox{\rm Irr}\Lambda(\nu)\bigr)_{i}^c
\overset{\sim}{\to}
\bigl(\mbox{\rm Irr}\Lambda(\onu)\bigr)_{i}^0,$$
respectively. We introduce maps 
$$\te_i,\te_i^*:\bigsqcup_{\nu\in Q_+}\mbox{Irr}\Lambda(\nu)\to
\bigsqcup_{\nu\in Q_+}\mbox{Irr}\Lambda(\nu)\sqcup\{0\}\quad\mbox{and}\quad
\tf_i,\tf_i^*:\bigsqcup_{\nu\in Q_+}\mbox{Irr}\Lambda(\nu)\to
\bigsqcup_{\nu\in Q_+}\mbox{Irr}\Lambda(\nu)$$
as follows: if $c>0$ we define 
$$\begin{array}{llllllllllll}
\te_i:& \bigl(\mbox{\rm Irr}\Lambda(\nu)\bigr)_{i,c} & 
\overset{\sim}{\longrightarrow} & 
\bigl(\mbox{\rm Irr}\Lambda(\nu')\bigr)_{i,0} & 
\overset{\sim}{\longrightarrow} & 
\bigl(\mbox{\rm Irr}\Lambda(\nu+\alpha_i)\bigr)_{i,c-1},\\
\te_i^*:& \bigl(\mbox{\rm Irr}\Lambda(\nu)\bigr)_i^c & 
\overset{\sim}{\longrightarrow} & 
\bigl(\mbox{\rm Irr}\Lambda(\onu)\bigr)_i^0 & 
\overset{\sim}{\longrightarrow} & 
\bigl(\mbox{\rm Irr}\Lambda(\nu+\alpha_i)\bigr)_i^{c-1}
\end{array}$$
and $\te_i\Lambda=0$ and $\te_i^*\Lambda'=0$ for 
$\Lambda\in\bigl(\mbox{\rm Irr}\Lambda(\nu)\bigr)_{i,0}$ and
$\Lambda'\in\bigl(\mbox{\rm Irr}\Lambda(\nu)\bigr)_{i}^0$, respectively. 
Define
$$\begin{array}{llllllllllll}
\tf_i:& \bigl(\mbox{\rm Irr}\Lambda(\nu)\bigr)_{i,c} & 
\overset{\sim}{\longrightarrow} & 
\bigl(\mbox{\rm Irr}\Lambda(\nu')\bigr)_{i,0} & 
\overset{\sim}{\longrightarrow} & 
\bigl(\mbox{\rm Irr}\Lambda(\nu-\alpha_i)\bigr)_{i,c+1},\\
\tf_i^*:& \bigl(\mbox{\rm Irr}\Lambda(\nu)\bigr)_i^c & 
\overset{\sim}{\longrightarrow} & 
\bigl(\mbox{\rm Irr}\Lambda(\onu)\bigr)_i^0 & 
\overset{\sim}{\longrightarrow} & 
\bigl(\mbox{\rm Irr}\Lambda(\nu-\alpha_i)\bigr)_i^{c+1}.
\end{array}$$
\begin{thm}\cite{KS}\label{thm:KS}
{\rm (1)} 
For $\Lambda\in \mbox{\em Irr}\Lambda(\nu)$, we set
$\mbox{\rm wt}\Lambda=-\nu$, $\vphi_i(\Lambda)=\eps_i(\Lambda)+
\langle h_i,\mbox{\rm wt}\Lambda\rangle$.
Then
$(\bigsqcup_{\nu\in Q_+}\mbox{\em Irr}\Lambda(\nu);\mbox{\rm wt}, \eps_i, 
\vphi_i, \te_i, \tf_i)$ is a crystal isomorphic to
$(B(\infty);\mbox{\rm wt}, \eps_i, \vphi_i, \te_i, \tf_i)$.
More precisely, the explicit form of the isomorphism 
$\Phi:B(\infty)\overset{\sim}{\to}\bigsqcup_{\nu\in Q_+}\mbox{\em Irr}
\Lambda(\nu)$ 
is given by $\Phi=\Psi_{\bf i}\circ\Xi_{\bf i}^{-1}$.  
\\
{\rm (2)} Set $\vphi_i^*(\Lambda)=\eps_i^*(\Lambda)+
\langle h_i,\mbox{\rm wt}\Lambda\rangle$. Then 
$(\bigsqcup_{\nu\in Q_+}\mbox{\em Irr}\Lambda(\nu);\mbox{\rm wt}, \eps_i^*, 
\vphi_i^*, \te_i^*, \tf_i^*)$ is a crystal and the bijection $\Phi$ gives
an isomorphism of crystals form $(B(\infty);\mbox{\rm wt}, \eps_i^*, 
\vphi_i^*, \te_i^*, \tf_i^*)$ to it.
\end{thm}
\begin{rem}{\rm
Because of (7.1.1), the definition of the map
$\Phi:B(\infty)\overset{\sim}{\to}\bigsqcup_{\nu\in Q_+}\mbox{Irr}\Lambda(\nu)$
is independent of the choice of ${\bf i}$.
}\end{rem}
\subsection{A proof of Lemma \ref{lemma:isom2}}
We only show the second formula:
$$\tf_i^*({\bf M}({\bf a}))={\bf M}(\tf_i^*{\bf a})\quad({\bf a}\in\cB),$$ 
because the first one is proved by similar method. \\

By Corollary \ref{cor:ast-crys}, it is enough to show that
$$M_{[1,i]^c}(\tf_i^*{\bf a})=M_{[1,i]^c}({\bf a})-1
,\eqno{(7.3.1)}$$
$$M_K(\tf_i^*{\bf a})=M_{K}({\bf a})\mbox{ for all }K\in\cM_n^{\times}
\setminus\cM_n^{\times}(i)^*.\eqno{(7.3.2)}$$
It is easy to see (7.3.1). Indeed, as in the proof of Lemma \ref{lemma:isom1}, 
we have
$$
\mbox{wt}({\bf a})=\mbox{wt}({\bf M}({\bf a}))
=\sum_{i\in I}M_{[1,i]^c}({\bf a})\alpha_i.
$$ 
Similarly we have
$$\mbox{wt}(\tf_i^*{\bf a})=\sum_{i\in I}M_{[1,i]^c}(\tf_i^*{\bf a})\alpha_i.$$
Since $\cB$ is a $\ast$-crystal, we have $\mbox{wt}(\tf_i^*{\bf a})
=\mbox{wt}({\bf a})-\alpha_i.$ Therefore (7.3.1) holds.\\

We shall prove (7.3.2). Assume $K\in\cM_n^{\times}\setminus\cM_n^{\times}(i)^*.$
Namely, $i\in K$ or $i+1\not\in K$. By (7.1.2) and Theorem \ref{thm:KS},
it is enough to prove that $M_K(\tf_i^*\Lambda_{{\bf a}})=M_K(\Lambda_{{\bf a}})$
for $i\in K$ or $i+1\not\in K$. Moreover, since $\te_i^{\ast max}
(\tf_i^*\Lambda_{\bf a})=\te_i^{\ast max}\Lambda_{\bf a}$, it is enough to show 
that
$$M_K(\te_i^{\ast max}\Lambda_{{\bf a}})=M_K(\Lambda_{{\bf a}})\qquad
(i\in K\mbox{ or }i+1\not\in K).\eqno{(7.3.3)}$$
 
Assume $i\in K_m=[s_m+1,t_m]~\subset K$. Then, there are the following three 
cases; (a) $m=1$ and $s_1=1$, (b) $m=l$ and $t_m=n+1$, (c) otherwise. 
In the case (a), there is the 
path $\sigma(1\leftarrow t_1)$ from $t_1$ to $1$ in 
$\Omega(K)$ which is is trough $i$. 
Since $1$ (the end point of this path) is not an element of $\inn(K)$, 
this path does not appear in the definition of $M_K(\Lambda)$ for any 
$\Lambda$. Therefore $M_K(\te_i^{\ast max}\Lambda_{{\bf a}})
=M_K(\Lambda_{{\bf a}})$. By the similar way, we have (7.3.3) in the case (b).

Let us consider the case (c). In this case, 
the path $\sigma(s_m\leftarrow t_m)$ in $\Omega(K)$ is trough $i$ and 
$s_m\in\inn(K)$, $t_m\in\out(K)$. We remark that $s_m<i$ since $i\in K$.
For simplicity, we denote $\Lambda=\Lambda_{{\bf a}}$ and $\overline{\Lambda}=
\te_i^{\ast max}\Lambda_{{\bf a}}$. Let $\nu=\mbox{wt}(\Lambda)$ and $\onu
=\mbox{wt}(\overline{\Lambda})$, respectively.
Take a general point
$\overline{B}=(\overline{B}_{\tau})_{\tau\in H}\in\overline{\Lambda}$. Recall the
diagram (7.2.3) and take a general point $B=(B_{\tau})_{\tau\in H}\in\Lambda(\nu)$
of $q_2\circ q_1^{-1}(B)$. Then $B$ is a general point of $\Lambda$. 
By the constriction we have the following commutative diagram:
\begin{center}
\setlength{\unitlength}{1mm}
\begin{picture}(50,40)
\put(5,30){$V_{t_m}(\nu)$}\put(19,30){$\overset{\sim}{\to}$}
\put(25,30){$V_{t_m}(\onu)$}
\put(9,23){$\downarrow$}\put(29,23){$\downarrow$}
\put(6,16){$V_i(\nu)$}\put(19,16){$\twoheadrightarrow$}
\put(26,16){$V_i(\onu)$}
\put(9,9){$\downarrow$}\put(29,9){$\downarrow$}
\put(5,2){$V_{s_m}(\nu)$}\put(19,2){$\overset{\sim}{\to}$}
\put(25,2){$V_{s_m}(\onu)$.}
\put(-5,23.5){\scriptsize{$B_{\sigma(i\leftarrow t_m)}$}}
\put(33,23.5){\scriptsize{$\overline{B}_{\sigma(i \leftarrow t_m)}$}}
\put(-5,9.5){\scriptsize{$B_{\sigma(s_m\leftarrow i)}$}}
\put(33,9.5){\scriptsize{$\overline{B}_{\sigma(s_m\leftarrow i)}$}}
\put(19.5,36){\scriptsize{$\ophi_{t_m}$}}
\put(19.5,20){\scriptsize{$\ophi_{i}$}}
\put(19.5,8){\scriptsize{$\ophi_{s_m}$}}
\end{picture}
\end{center}
Therefore we have
$$\mbox{Im}(B_{\sigma(s_m\leftarrow t_m)})=
\mbox{Im}(\overline{B}_{\sigma(s_m\leftarrow t_m)})$$
and this formula tells us (7.3.3) holds.\\

For the case of $i+1\not\in K$, we have (7.3.3) by the similar method. Thus we
complete a proof of Lemma \ref{lemma:isom2}.
\section{A new proof of the Anderson-Mirkovi\'c conjecture}
\subsection{Reformulation of the Anderson-Mirkovi\'c conjecture}
Let us denote $\Lambda=\Lambda_{\bf a}$. Then Corollary \ref{cor:AM-e} can be 
written as
$$(\tf_i^*{\bf M}(\Lambda))_K=\left\{\begin{array}{ll}
\mbox{min}\left\{M_K(\Lambda),~M_{s_iK}(\Lambda)+c_i^*({\bf M}(\Lambda))
\right\} &
(K\in \cM_n^{\times}(i)^*),\\
M_K(\Lambda) & (\mbox{otherwise}).
\end{array}\right.$$
Here $c_i^*({\bf M}(\Lambda))=M_{[1,i]^c}(\Lambda)-
M_{([1,i+1]\setminus\{i\})^c}(\Lambda)-1$.
By Lemma \ref{lemma:isom2}, we already know that
$$(\tf_i^*{\bf M}(\Lambda))_K=M_K(\tf_i^*\Lambda)\quad\mbox{ for }
K\in\cM_n^{\times}.$$
Moreover, by (7.3.2), we have
$$(\tf_i^*{\bf M}(\Lambda))_K=M_K(\Lambda)\quad\mbox{ for }K\in 
\cM_n^{\times}\setminus\cM_n^{\times}(i)^*.$$
Therefore it is enough to show
$$M_K(\tf_i^*\Lambda)=
\mbox{min}\left\{M_K(\Lambda),~M_{s_iK}(\Lambda)+c_i^*({\bf M}(\Lambda))\right\}
\mbox{ for }K\in \cM_n^{\times}(i)^*.\eqno{(8.1.1)}$$
\begin{lemma}\label{lemma:AM-e}
The formula $(8.1.1)$ is equivalent to the following: 
$$M_K(\Lambda)=\mbox{\em min}\{M_K(\overline{\Lambda}),~
M_{s_iK}(\overline{\Lambda})
+\langle h_i,\mbox{\em wt}(\overline{\Lambda})\rangle-\eps_i^*(\Lambda)\}
\mbox{ for }
K\in \cM_n^{\times}(i)^*.
\eqno{(8.1.2)}$$
Here $\overline{\Lambda}=\te_i^{\ast max}\Lambda$.
\end{lemma}
\begin{rem}{\rm
The formula (8.1.2) is a generalization of the formula which appears in our
previous paper \cite{KS}. 

Consider the case of $K=[1,i+1]\setminus\{i\}~(1\leq i\leq n)$. Then we have
$$\out(K)=\left\{\begin{array}{ll}
\{2\} & (i=1),\\
\{i-1,i+1\} & (2\leq i\leq n-1),\\
\{n-1\} & (i=n)
\end{array}\right.\quad\mbox{and}\quad
\inn(K)=\{i\}.$$
Therefore we have
\begin{align*}
M_{[1,i+1]\setminus\{i\}}(\Lambda)&
=-\dimc\mbox{\rm Coker}\left(\mathop{\oplus}_{\tau;\sinn(\tau)=i}
V(\nu)_{\sout(\tau)}
\overset{\oplus B_{\tau}}{\longrightarrow}V(\nu)_i\right)\\
&=-\eps_i(\Lambda).
\end{align*}
Here $B=(B_{\tau})$ is a general point of $\Lambda$. 
Since $s_iK=[1,i]$, the formula (8.1.2) is equivalent to
$$\eps_i(\Lambda)=\mbox{max}\bigl\{\eps_i(\overline{\Lambda}), 
-\langle h_i,\mbox{wt}(\overline{\Lambda})\rangle+\eps_i^*(\Lambda)\bigr\}.$$
This is nothing but the formula which appears in \cite{KS}, Proposition 5.3.1,
(1).
}\end{rem}

\noindent
{\it Proof of Lemma \ref{lemma:AM-e}}.
Let $\nu$, $\onu$, $B$, $\overline{B}$ be as same as in the proof of Lemma 
\ref{lemma:isom2}. For simplicity, we denote 
$\oplus_{i\in I}V_i=\oplus_{i\in I}V_i(\nu)$ and
$\oplus_{i\in I}\overline{V}_i=\oplus_{i\in I}V_i(\onu)$. \\

Before proving the equivalence, we shall show
$$c_i^*({\bf M}(\Lambda))=\langle h_i,\mbox{wt}(\overline{\Lambda})\rangle
-\varepsilon_i^*(\Lambda)-1.\eqno{(8.1.3)}$$
Since
$$M_{[1,i]^c}(\Lambda)=-\dimc V_i,$$ 
$$M_{([1,i+1]\setminus\{i\})^c}(\Lambda)=
-\dimc\mbox{Coker}\left(V_i\to \mathop{\oplus}_{\tau~;~\mbox{out}(\tau)=i}
V_{\mbox{in}(\tau)}\right),$$
we have
\begin{align*}
&M_{[1,i]^c}(\Lambda)-M_{([1,i+1]\setminus\{i\})^c}(\Lambda)\\
&\qquad=-\dimc V_i
+\dimc\mbox{Coker}\left(V_i\to \mathop{\oplus}_{\tau~;~\mbox{out}(\tau)=i}
V_{\mbox{in}(\tau)}\right)\\
&\qquad
=-2\dimc V_i+\dimc\left(\mathop{\oplus}_{\tau~;~\mbox{out}(\tau)=i}
V_{\mbox{in}(\tau)}
\right)+\dimc\mbox{Ker}\left(V_i\to \mathop{\oplus}_{\tau~;~\mbox{out}(\tau)=i}
V_{\mbox{in}(\tau)}\right)\\
&\qquad=
\langle h_i,\mbox{wt}(\Lambda)\rangle
+\varepsilon_i^*(\Lambda).
\end{align*}
Moreover, since
$$\dimc\overline{V}_k=\left\{\begin{array}{ll}
\dimc V_k & (k\ne i),\\
\dimc V_i-\eps_i^*(\Lambda)&(k=i),
\end{array}\right.$$
we have
\begin{align*}
c_i^*({\bf M}(\Lambda))&=
\langle h_i,\mbox{wt}(\Lambda)\rangle
+\varepsilon_i^*(\Lambda)-1\\
&=\langle h_i,\mbox{wt}(\overline{\Lambda})-
\varepsilon_i^*(\Lambda)\alpha_i\rangle
+\varepsilon_i^*(\Lambda)-1\\
&=\langle h_i,\mbox{wt}(\overline{\Lambda})\rangle
-\varepsilon_i^*(\Lambda)-1.
\end{align*}
Thus, (8.1.3) is proved.\\

Let us prove the equivalence. Firstly, we will show (8.1.1) $\Rightarrow$ 
(8.1.2). 
Applying (8.1.1) for $\Lambda_1=\widetilde{e}_i^{*}\Lambda$, we have
\begin{align*}
M_K(\Lambda)&=M_K(\widetilde{f}_i^*\Lambda_1)\\
&=\mbox{min}\left\{M_K(\Lambda_1),~
M_{s_iK}(\Lambda_1)+c_i^*({\bf M}(\Lambda_1))\right\}.
\end{align*}
Since the vertex $i$ is a source in $\Omega(s_iK)$ and 
$\overline{\Lambda_1}=\overline{\Lambda}$, we have
$$M_{s_iK}(\Lambda_1)=M_{s_iK}(\overline{\Lambda_1})
=M_{s_iK}(\overline{\Lambda}).$$
On the other hand, 
\begin{align*}
c_i^*({\bf M}(\Lambda_1))
&=\langle h_i,\mbox{wt}(\overline{\Lambda_1})\rangle
-\varepsilon_i^*(\Lambda_1)-1\\
&=\langle h_i,\mbox{wt}(\overline{\Lambda})\rangle
-\varepsilon_i^*(\Lambda).
\end{align*}
Therefore we have
$$M_K(\Lambda)=\mbox{min}\left\{M_K(\Lambda_1),~
M_{s_iK}(\overline{\Lambda})+
\langle h_i,\mbox{wt}(\overline{\Lambda})\rangle
-\varepsilon_i^*(\Lambda)\right\}.\eqno{(8.1.4)}$$
Similarly, applying (8.1.1) for $\Lambda_2=\widetilde{e}_i^{*}\Lambda_1=
(\widetilde{e}_i^{*})^2\Lambda$, we have
\begin{align*}
M_K(\Lambda_1)&=\mbox{min}\left\{M_K(\Lambda_2),~
M_{s_iK}(\overline{\Lambda})+
\langle h_i,\mbox{wt}(\overline{\Lambda})\rangle
-\varepsilon_i^*(\Lambda_1)\right\}\\
&=\mbox{min}\left\{M_K(\Lambda_2),~
M_{s_iK}(\overline{\Lambda})+
\langle h_i,\mbox{wt}(\overline{\Lambda})\rangle
-\varepsilon_i^*(\Lambda)+1\right\}.
\end{align*}
By substituting this formula for (8.1.4), we have
\begin{align*}
M_K(\Lambda)&=\mbox{min}\left\{\mbox{min}\left\{M_K(\Lambda_2),~
M_{s_iK}(\overline{\Lambda})+
\langle h_i,\mbox{wt}(\overline{\Lambda})\rangle
-\varepsilon_i^*(\Lambda)+1\right\}\right.,\\
&\hspace*{50mm}~\left.
M_{s_iK}(\overline{\Lambda})+
\langle h_i,\mbox{wt}(\overline{\Lambda})\rangle
-\varepsilon_i^*(\Lambda)\right\}\\
&=\mbox{min}\left\{M_K(\Lambda_2),~
M_{s_iK}(\overline{\Lambda})+
\langle h_i,\mbox{wt}(\overline{\Lambda})\rangle
-\varepsilon_i^*(\Lambda)\right\}.
\end{align*}
After repeating the similar method, we have
$$M_K(\Lambda)=\mbox{min}\left\{M_K(\overline{\Lambda}),~
M_{s_iK}(\overline{\Lambda})+
\langle \alpha_i,\mbox{wt}(\overline{\Lambda})\rangle
-\varepsilon_i^*(\Lambda)\right\}.$$
This is nothing but the formula (8.1.2).\\

Secondly let us prove (8.1.2) $\Rightarrow$ (8.1.1). By (8.1.2) for 
$\widetilde{f}_i^*\Lambda$ and $\overline{\widetilde{f}_i^*\Lambda}=
\overline{\Lambda}$, we have
\begin{align*}
M_K(\widetilde{f}_i^*\Lambda)
&=\mbox{min}\left\{M_K(\overline{\widetilde{f}_i^*\Lambda}),
M_{s_iK}(\overline{\widetilde{f}_i^*\Lambda})
+\langle \alpha_i,\mbox{wt}(\overline{\widetilde{f}_i^*\Lambda})\rangle
-\varepsilon_i^*(\widetilde{f}_i^*\Lambda)\right\}\\
&=\mbox{min}\left\{
M_K(\overline{\Lambda}),
M_{s_iK}(\overline{\Lambda})
+\langle h_i,\mbox{wt}(\overline{\Lambda})\rangle
-\varepsilon_i^*(\Lambda)-1\right\}.
\end{align*}
Since $M_{s_iK}(\overline{\Lambda})
+\langle h_i,\mbox{wt}(\overline{\Lambda})\rangle
-\varepsilon_i^*(\Lambda)>M_{s_iK}(\overline{\Lambda})
+\langle h_i,\mbox{wt}(\overline{\Lambda})\rangle
-\varepsilon_i^*(\Lambda)-1$, 
\begin{align*}
\mbox{the right hand side}&=
\mbox{min}\left\{\mbox{min}\left\{M_K(\overline{\Lambda}),
M_{s_iK}(\overline{\Lambda})
+\langle h_i,\mbox{wt}(\overline{\Lambda})\rangle
-\varepsilon_i^*(\Lambda)\right\},\right.\\
&\qquad\qquad\qquad\qquad\qquad \left. M_{s_iK}(\overline{\Lambda})
+\langle h_i,\mbox{wt}(\overline{\Lambda})\rangle
-\varepsilon_i^*(\Lambda)-1\right\}\\
&=\mbox{min}\left\{M_K(\Lambda), M_{s_iK}(\overline{\Lambda})
+\langle h_i,\mbox{wt}(\overline{\Lambda})\rangle
-\varepsilon_i^*(\Lambda)-1\right\}\\
&=\mbox{min}\left\{M_K(\Lambda), M_{s_iK}(\overline{\Lambda})
+c_i^*({\bf M}(\Lambda))\right\}.
\end{align*}
Because $s_iK\in \cM_n^{\times}\setminus \cM_n^{\times}(i)^*$, we have
$M_{s_iK}(\overline{\Lambda})=M_{s_iK}(\Lambda)$ by (7.3.3). Therefore we have
$$M_K(\widetilde{f}_i^*\Lambda)= \mbox{min}\left\{M_K(\Lambda), 
M_{s_iK}({\Lambda})
+c_i^*({\bf M}(\Lambda))\right\}.$$
This is nothing but (8.1.1).
\hfill$\square$
\subsection{A proof of the formula (8.1.2)}
The aim of this subsection is to prove the next proposition: 
\begin{prop}\label{prop:new-AM}
The formula $(8.1.2)$ holds for any $K\in \cM_n^{\times}(i)^*$.
\end{prop}
Set
$$\overline{W}_{s_iK}=\mbox{Ker}\left(
\mathop{\oplus}_{p\in \mbox{out}(s_iK)}\overline{V}_p
\quad\overset{\oplus\overline{B}_{p\to q}}{\longrightarrow}
\mathop{\oplus}_{q\in \mbox{in}(s_iK)}\overline{V}_q\right)\subset
\left(\mathop{\oplus}_{p\in \mbox{out}(s_iK)}\overline{V}_p\right).$$
By the assumption, we have
$$\mbox{out}(s_iK)\setminus\{i\}\subset\out(K)\subset
\bigl(\mbox{out}(s_iK)\setminus\{i\}\bigr)\cup\{i-1,i+1\}.$$
Therefore we can define a map $\Phi:\overline{W}_{s_iK}\to
\left(\mathop{\oplus}_{k\in \mbox{out}(K)}V_k\right)
$ by
\begin{align*}
\Phi\left(\sum_{p\in \mbox{out}(s_iK)}\overline{w}_p\right)&=
\delta_K(i-1)\overline{B}_{i\to i-1}(\overline{w}_i)+
\delta_K(i+1)\overline{B}_{i\to i+1}(\overline{w}_i)
+\sum_{p\in \mbox{out}(s_iK)\setminus\{i\}}\overline{w}_p,
\end{align*}
where $\overline{w}_p\in\overline{V}_p$ and $\delta_K$ is a map form $[1,n]$
to $\{0,1\}$ defined by $$\delta_K(k)=\left\{\begin{array}{ll}
1 & (k\in \mbox{out}(K)),\\
0 & (\mbox{otherwise}).
\end{array}\right.$$
Here we remark that, if $p\in \mbox{out}(s_iK)\setminus\{i\}$, we have
$p\in\mbox{out}(K)$ and $\overline{V}_p=V_p$.\\

Set
$$N=\mbox{Coker}(\Phi)$$
and consider a map
$$\widetilde{\mbox{Id}}:\left(\mathop{\oplus}_{k\in \mbox{out}(K)}V_k\right)\to
N$$
which is naturally induced form the identity map
$\displaystyle{\mbox{Id}:\left(\mathop{\oplus}_{k\in \mbox{out}(K)}V_k\right)
\overset{\sim}{\to}\left(\mathop{\oplus}_{k\in \mbox{out}(K)}V_k\right)}.$
By the construction, it is clear that $\widetilde{\mbox{Id}}$ is surjective.\\

In the above setting, the following two Lemmas hold:
\begin{lemma}
$\Phi$ is injective.
\end{lemma}

\begin{lemma}
Let $k\in\mbox{\rm out}(K),~l\in\mbox{\rm in}(K)$ and assume that there is a 
path $\sigma(k\to l)$ in the orientation $\Omega(K)$.
\vskip 1mm
\noindent
{\rm (1)} If $l\ne i$, then the map $B_{\sigma(k\to l)}:V_k\to V_l~
(=\overline{V}_l)$ induces a map $\psi_l:N\to V_l$ such that
$$B_{\sigma(k\to l)}~(=\overline{B}_{\sigma(k\to l)})=
\psi_l\circ\widetilde{\mbox{Id}}.$$
{\rm (2)} If $l=i$, then the map $B_{\sigma(k\to i)}:V_k\to \overline{V_i}
~(\ne V_i)$ induces a map $\psi_i:N\to \overline{V}_i$ such that
$$\overline{B}_{\sigma(k\to i)}=\psi_i\circ\widetilde{\mbox{Id}}.$$
Moreover, let $\pi_i:V_i\to \overline{V}_i$ be the natural projection and 
$\varphi_i:N\to V_i$ a generic map such that $\psi_i=\pi_i\circ\varphi_i$. Then
we have
$$B_{\sigma(k\to i)}=\varphi_i\circ\widetilde{\mbox{Id}}.$$
\end{lemma}
The above two lemmas are easy exercises on linear algebra. So, we omit to give
proofs.\\
\\
\noindent
{\it Proof of Proposition \ref{prop:new-AM}.} Since $\Phi$ is injective, 
we have
\begin{align*}
\dimc N &=
\mathop{\sum}_{k\in\mbox{out}(K)}\dimc V_k-\dimc \overline{W}_{s_iK}\\
&=\mathop{\sum}_{k\in\mbox{out}(K)}\dimc V_k-\dimc \mbox{Ker}\left(
\mathop{\oplus}_{p\in \mbox{out}(s_iK)}\overline{V}_p
\quad\overset{\oplus\overline{B}_{\sigma(p\to q)}}{\longrightarrow}
\mathop{\oplus}_{q\in \mbox{in}(s_iK)}\overline{V}_q\right)\\
&=\mathop{\sum}_{k\in\mbox{out}(K)}\dimc V_k-
\mathop{\sum}_{p\in\mbox{out}(s_iK)}\dimc \overline{V}_p
+\mathop{\sum}_{q\in\mbox{in}(s_iK)}\dimc \overline{V}_q\\
&\qquad\qquad 
-\dimc\mbox{Coker}\left(
\mathop{\oplus}_{p\in \mbox{out}(s_iK)}\overline{V}_p
\quad\overset{\oplus\overline{B}_{\sigma(p\to q)}}{\longrightarrow}
\mathop{\oplus}_{q\in \mbox{in}(s_iK)}\overline{V}_q\right)\\
&=\mathop{\sum}_{k\in\mbox{out}(K)}\dimc V_k-
\mathop{\sum}_{p\in\mbox{out}(s_iK)}\dimc \overline{V}_p
+\mathop{\sum}_{q\in\mbox{in}(s_iK)}\dimc \overline{V}_q
+M_{s_iK}(\overline{\Lambda}).
\end{align*}
Denote the direct sum of the maps $\psi_l$ ({\it resp.} $\varphi_l$) 
$(l\in\mbox{in(K)})$ by
$$\psi=\mathop{\oplus}_{l\in \mbox{in}(K)} \psi_l:
N\to \mathop{\oplus}_{l\in \mbox{in}(K)}\overline{V}_l
\qquad\left({\it resp.}~ \varphi=\mathop{\oplus}_{l\in \mbox{in}(K)}\varphi_l:
N\to \mathop{\oplus}_{l\in \mbox{in}(K)}
{V}_l\right).
$$
Here we set $\varphi_l=\psi_l$ for $l\ne i$.

By the definition, we have
$$\mbox{Im}\left(
\mathop{\oplus}_{k\in \mbox{out}(K)}{V}_k
\quad\overset{\oplus{B}_{\sigma(k\to l)}}{\longrightarrow}
\mathop{\oplus}_{l\in \mbox{in}(K)}{V}_l\right)
=\mbox{Im}\left(N\overset{\varphi}{\longrightarrow}
\mathop{\oplus}_{l\in \mbox{in}(K)}{V}_l\right).$$
Moreover, by the genericity of $\varphi$, we have
$$\dim\mbox{Ker}\varphi
=\mbox{max}\bigl\{\dimc\mbox{Ker}\psi-\varepsilon_i^*(\Lambda),0\bigr\}.$$
Combining the above results, we have
$$-M_K(\Lambda)=\mathop{\sum}_{l\in \mbox{in}(K)}\dimc{V}_l-\dimc N+
\mbox{max}\bigl\{\dimc\mbox{Ker}\psi-\varepsilon_i^*(\Lambda),0\bigr\}.
\eqno{(8.2.1)}$$
Indeed,
\begin{align*}
-M_K(\Lambda)&=\dimc\mbox{Coker}\left(
\mathop{\oplus}_{k\in \mbox{out}(K)}{V}_k
\quad\overset{\oplus{B}_{\sigma(k\to l)}}{\longrightarrow}
\mathop{\oplus}_{l\in \mbox{in}(K)}{V}_l\right)\\
&=\mathop{\sum}_{l\in \mbox{in}(K)}\dimc{V}_l-\dimc\mbox{Im}\left(
\mathop{\oplus}_{k\in \mbox{out}(K)}{V}_k
\quad\overset{\oplus{B}_{\sigma(k\to l)}}{\longrightarrow}
\mathop{\oplus}_{l\in \mbox{in}(K)}{V}_l\right)\\
&=\mathop{\sum}_{l\in \mbox{in}(K)}\dimc{V}_l-\dimc\mbox{Im}\varphi\\
&=\mathop{\sum}_{l\in \mbox{in}(K)}\dimc{V}_l-\dimc N+\dimc\mbox{Ker}\varphi\\
&=\mathop{\sum}_{l\in \mbox{in}(K)}\dimc{V}_l-\dimc N+
\mbox{max}\bigl\{\dimc\mbox{Ker}\psi-\varepsilon_i^*(\Lambda),0\bigr\}.
\end{align*}
\begin{lemma} The following formulas hold:
\vskip 1mm
\noindent
{\rm (1)} $\displaystyle{\mathop{\sum}_{l\in \mbox{\rm in}(K)}\dimc{V}_l-
\dimc N+\dimc\mbox{\rm Ker}\psi-\varepsilon_i^*(\Lambda)=
-M_{K}(\overline{\Lambda})},$
\vskip 1mm
\noindent
{\rm (2)} $\displaystyle{\mathop{\sum}_{l\in \mbox{\rm in}(K)}\dimc{V}_l-
\dimc N=
-M_{s_iK}(\overline{\Lambda})+\varepsilon_i^*(\Lambda)-\langle h_i,
\mbox{\rm wt}(\overline{\Lambda})\rangle}$.
\end{lemma}
\begin{proof}
The formula (1) is proved by a direct computation. Indeed, we have
\begin{align*}
&\mathop{\sum}_{l\in \mbox{in}(K)}\dimc{V}_l-\dimc N+\dimc\mbox{Ker}\psi
-\varepsilon_i^*(\Lambda)\\
&\qquad\qquad 
=\mathop{\sum}_{l\in \mbox{in}(K)}\dimc{\overline{V}}_l-\dimc\mbox{Im}\psi\\
&\qquad\qquad =\mathop{\sum}_{l\in \mbox{in}(K)}\dimc{\overline{V}}_l-
\dimc\mbox{Im}\left(
\mathop{\oplus}_{k\in \mbox{out}(K)}\overline{V}_k
\quad\overset{\oplus\overline{B}_{\sigma(k\to l)}}{\longrightarrow}
\mathop{\oplus}_{l\in \mbox{in}(K)}\overline{V}_l\right)\\
&\qquad\qquad =\dimc\mbox{Coker}\left(
\mathop{\oplus}_{k\in \mbox{out}(K)}\overline{V}_k
\quad\overset{\oplus\overline{B}_{\sigma(k\to l)}}{\longrightarrow}
\mathop{\oplus}_{l\in \mbox{in}(K)}\overline{V}_l\right)\\
&\qquad\qquad =-M_{K}(\overline{\Lambda}).
\end{align*}
Let us show the formula (2). We have
\begin{align*}
\mathop{\sum}_{l\in \mbox{in}(K)}\dimc{V}_l-\dimc N
&=\mathop{\sum}_{l\in \mbox{in}(K)}\dimc{V}_l
-\mathop{\sum}_{k\in \mbox{out}(K)}\dimc{V}_k\\
&\qquad\qquad+\mathop{\sum}_{p\in \mbox{out}(s_iK)}\dimc\overline{V}_p
-\mathop{\sum}_{q\in \mbox{in}(s_iK)}\dimc\overline{V}_q
-M_{s_iK}(\overline{\Lambda}).
\end{align*}
Therefore it is enough to show that 
\begin{align*}
&\mathop{\sum}_{l\in \mbox{in}(K)}\dim{V}_l
-\mathop{\sum}_{k\in \mbox{out}(K)}\dim{V}_k
+\mathop{\sum}_{p\in \mbox{out}(s_iK)}\dim\overline{V}_p
-\mathop{\sum}_{q\in \mbox{in}(s_iK)}\dim\overline{V}_q\\
&\qquad\quad=\varepsilon_i^*(\Lambda)-\langle h_i,
\mbox{wt}(\overline{\Lambda})\rangle.
\end{align*}
But it is easily checked by a direct computation. Thus, we get the formula.
\end{proof}
Let us return to the proof of Proposition \ref{prop:new-AM}.
Substituting the result of the above lemma for (8.2.1), we get
$$-M_K(\Lambda)=\mbox{max}\left\{-M_K(\overline{\Lambda}),~
-M_{s_iK}(\overline{\Lambda})
-\langle \alpha_i,\mbox{wt}(\overline{\Lambda})\rangle
+\varepsilon_i^*(\Lambda)\right\}.$$
This is nothing but the formula (8.1.2). Thus we have the statement.
\hfill$\square$
\appendix
\section{A proof of the formula (7.1.1)}
In this appendix, we will give a proof of the formula (7.1.1).

\subsection{Results in  \cite{KS}}
Let us recall some results in \cite{KS}.
Consider two quivers of type $A_n$, $(I,\Omega)$ and $(I,\Omega')$. 
Assume ${\bf i}$ ({\it resp}. ${\bf i}'$) is 
a reduced word of the longest element which is adapted to the orientation
$\Omega$ ({\it resp}. $\Omega'$).  

Let $\nu\in Q_+$ and $\mathcal{O}_{{\bf a},\Omega}$ be the $G(\nu)$-orbit in 
$E_{V(\nu),\Omega}$ corresponding to an ${\bf i}$-Lusztig datum ${\bf a}$, and 
$\mathcal{O}_{{\bf a}',\Omega'}$ the orbit in 
$E_{V(\nu),\Omega'}$ corresponding to an ${\bf i}'$-Lusztig datum 
${\bf a'}=R_{\bf i}^{{\bf i}'}({\bf a})$ where $R_{\bf i}^{{\bf i}'}$ is the 
transition map from ${\bf i}$ to ${\bf i}'$. 

Set 
$$\Lambda_{\bf a}:=\overline{T^*_{\mathcal{O}_{\bf a},\Omega}E_{V(\nu),\Omega}}
\quad\mbox{and}\quad
\Lambda_{{\bf a}'}:=\overline{T^*_{\mathcal{O}_{{\bf a}'},\Omega'}E_{V(\nu),\Omega'}}.
$$
The goal is to prove
$$\Lambda_{\bf a}=\Lambda_{{\bf a}'}.\eqno{(A.1.1)}$$
\vskip 3mm
Since our quiver is of type $A_n$, there uniquely exists a $G(\nu)$-orbit 
$\mathcal{O}_{{\bf b},\Omega}$ in $E_{V(\nu),\Omega}$ such that 
$\Lambda_{{\bf a}'}=\overline{T^*_{\mathcal{O}_{{\bf b},\Omega}}E_{V(\nu),\Omega}}$. 
(Here we denote by ${\bf b}$
the corresponding ${\bf i}$-Lusztig datum.)
We consider a map 
$$s:\{G(\nu)\mbox{-orbits in }E_{V(\nu),\Omega}\}\to
\{G(\nu)\mbox{-orbits in }E_{V(\nu),\Omega}\}$$
defined by $\mathcal{O}_{{\bf a},\Omega}\mapsto \mathcal{O}_{{\bf b},\Omega}$.
We remark that, by the definition, $s$ is 
a bijection.
To prove the formula $(A.1.1)$ (or equivalently (7.1.1)), it is enough to 
show that $s$ is the identity map. \\

Let $\nc_{\mathcal{O}_{{\bf a},\Omega}}$ be the constant sheaf on the orbit 
$\mathcal{O}_{{\bf a},\Omega}$, ${}^{\pi}\nc_{\mathcal{O}_{{\bf a},\Omega}}$ its 
minimal extension and 
$SS({}^{\pi}\nc_{\mathcal{O}_{{\bf a},\Omega}})$ its singular support. Then we have
$$SS({}^{\pi}\nc_{\mathcal{O}_{{\bf a},\Omega}})\supset 
\Lambda_{\bf a}=\overline{T^*_{\mathcal{O}_{{\bf a},\Omega}}E_{V(\nu),\Omega}}.$$
In addition to the above, the next result is proved by Lusztig: 
\begin{thm}[\cite{L1},\cite{L3}]
$$SS({}^{\pi}\nc_{\mathcal{O}_{{\bf a},\Omega}})
=SS({}^{\pi}\nc_{\mathcal{O}_{{\bf a}',\Omega'}}).$$
\end{thm}
Therefore we have
$$SS({}^{\pi}\nc_{\mathcal{O}_{{\bf a},\Omega}})
=SS({}^{\pi}\nc_{\mathcal{O}_{{\bf a}',\Omega'}})\supset 
\Lambda_{{\bf a}'}=
\overline{T^*_{\mathcal{O}_{{\bf b},\Omega}}E_{V(\nu),\Omega}}=
\overline{T^*_{s(\mathcal{O}_{{\bf a},\Omega})}E_{V(\nu),\Omega}}.\eqno{(A.1.2)}$$
\begin{rem}{\rm
In \cite{L3} and \cite{KS}, they consider a similar problem in more general 
setting. 
Let $\gtg$ be an arbitrary symmetric Kac-Moody Lie algebra and $B(\infty)$
the crystal basis of $U_q^-(\gtg)$. In \cite{KS}, they define a crystal
structure on the set of all irreducible components of Lusztig's quiver 
varieties, and show that it is isomorphic to $B(\infty)$. 
For $b\in B(\infty)$, we denote by $\Lambda_b$ the corresponding 
irreducible component.
On the other hand, let ${\bf B}$ be the Lusztig's canonical 
basis of $U_q^-(\gtg)$. It is constructed as the set of certain simple perverse
shaeves on the sepace of representations of a quiver attached to $\gtg$ 
(\cite{L3}).
Recall that there is a canonical bijection between $B(\infty)$ and 
${\bf B}$. For $b\in B(\infty)$, we denote by $L_{b,\Omega}$ the corresponding 
simple perverse sheaf. Then it is known that
$$SS(L_{b,\Omega})\supset \Lambda_{b}\quad \mbox{for any }b\in B(\infty).$$

Now, let us consider the following problem:\\
\\
{\bf Problem}. Assume $s:B(\infty)\to B(\infty)$ 
is a bijection such that $SS(L_{b,\Omega})\supset \Lambda_{s(b)}$ for any
$b\in B(\infty)$. Then, is $s$ the identity?\\

The above problem was firstly considered by Lusztig (\cite{L3}). 
In \cite{KS}, they stated that ``$s$ must be the identity under the above 
assumption''. But their ``proof'' of the statement is wrong. Therefore, the
problem is still open for an arbitrary case.

On the other hand, Kimura \cite{Kim} shows that the bijection $s$ must be the 
identity for finite $ADE$ type cases and cyclic quiver cases. 
In the next subsection, we will give a proof following \cite{Kim}.
}\end{rem}
\subsection{Kimura's proof of (A.1.1)}
By $(A.1.2)$ and the definition of singular supports, we have
$$\overline{\mathcal{O}_{{\bf a},\Omega}}\supset \mathcal{O}_{{\bf b},\Omega}=
s(\mathcal{O}_{{\bf a},\Omega}).$$
Namely, the bijection $s$ preserves closure relations on $G(\nu)$-orbits
in $E_{V(\nu),\Omega}$:
$$s({\mathcal{O}_{\Omega}})\subset
\overline{\mathcal{O}_{\Omega}}
\quad\mbox{for any $G(\nu)$-
orbit }\mathcal{O}_{\Omega}\mbox{ in }E_{V(\nu),\Omega}.\eqno{(A.2.1)}$$
Hence, it is enough to prove that the following statement:
\begin{prop}[\cite{Kim}]\label{prop:Kim}
Let $s$ be a bijection on the set of $G(\nu)$-orbits
in $E_{V(\nu),\Omega}$ which preserves closure relations. Then $s$ must be
the identity.
\end{prop}

Let us introduce an ordering $\leq$ on the set of $G(\nu)$-orbits
in $E_{V(\nu),\Omega}$ by
$$\mathcal{O}_{\Omega}\leq \mathcal{O}_{\Omega}'\quad
\overset{{\rm def}.}{\Longleftrightarrow}\quad
\mathcal{O}_{\Omega}'\subset \overline{\mathcal{O}_{\Omega}}.$$
By using this ordering, $(A.2.1)$ can be rewritten as:
$$s(\mathcal{O}_{\Omega})\subset \overline{\mathcal{O}_{\Omega}}=
\mathcal{O}_{\Omega}\cup
\bigcup_{\mathcal{O}_{\Omega}':\mathcal{O}_{\Omega}< \mathcal{O}_{\Omega}'}
\mathcal{O}_{\Omega}'.\eqno{(A.2.2)}$$
\noindent
{\it Proof of \ref{prop:Kim}}.
We will show the statement by the decreasing induction on $\leq$. 
Note that, by the definition, there is a maximal element with respect to the 
ordering $\leq$. Let $\mathcal{O}_{\Omega}$ be a such element. 
Hence, by $(A.2.2)$ and the maximality of $\mathcal{O}_{\Omega}$, 
$s(\mathcal{O}_{\Omega})$ must be equal to $\mathcal{O}_{\Omega}$.
Assume $s(\mathcal{O}_{\Omega}')=
\mathcal{O}_{\Omega}'$ for any $\mathcal{O}_{\Omega}< \mathcal{O}_{\Omega}'$.
Since $s$ is a bijection, we have $s(\mathcal{O}_{\Omega})=\mathcal{O}_{\Omega}$
by $(A.2.2)$. \hfill$\square$
\begin{rem}{\rm
For finite $D,E$ cases and cyclic quiver cases, the similar method does work.
Namely, we can show that $s$ must be the identity for such cases (see 
\cite{Kim}, in detail). \\
}\end{rem}

\end{document}